%% file: Arxiv.tex
\documentclass[11pt]{article}

\date{}

\usepackage[utf8]{inputenc}
\usepackage[T1]{fontenc}

\usepackage{amsmath, amssymb, amsthm, mathtools, bm}
\usepackage{tensor}
\usepackage{dsfont}
\usepackage{autobreak}
\usepackage{nicefrac}
\usepackage{bbm}
\input{math_commands.tex}

\usepackage{thm-restate}

\usepackage{graphicx}
\usepackage{caption}
\usepackage{subcaption}
\usepackage{float}
\usepackage{rotating}
\usepackage{multirow}
\usepackage{booktabs}

\usepackage[ruled,vlined,linesnumbered]{algorithm2e}
\usepackage[noend]{algpseudocode}
\SetKwProg{Init}{Init}{}{}
\SetKw{KwInit}{Initialization}

\usepackage[dvipsnames]{xcolor}
\usepackage[colorlinks=true,allcolors=blue]{hyperref}

\usepackage{natbib}

\usepackage{xspace}
\usepackage{microtype}
\usepackage{comment}

\usepackage[toc]{appendix}

\usepackage[colorinlistoftodos]{todonotes}

\newcommand{\av}[1]{\textbf{\textcolor{blue}{(Abhi: #1)}}}

\title{Beyond First-Order Methods for $\ell_p$-Structured\\Non-Monotone Variational Inequalities}
\author{
  Abhijeet Vyas\\
  Purdue University\\
  \texttt{vyas26@purdue.edu}
  \and
  Brian Bullins\\
  Purdue University\\
  \texttt{bbullins@purdue.edu}
}

\input{macrosarxiv}
\makeatletter
\renewcommand{\eqref}[1]{(\ref{#1})}
\makeatother

\begin{document}
\maketitle
\begin{abstract}
We propose novel high-order algorithms for a class of $\ell_p$-structured non-monotone variational inequalities. In particular, work by Diakonikolas et al. (2021), which introduced the weak Minty variational inequality (weak-MVI)  setting, showed how to find an approximate first-order Euclidean stationary point for a strictly positive range of the weak-MVI parameter $\rho$. However, for the $\ell_p$-norm stationary point setting ($p \neq 2$), their guarantees are limited to $\rho=0$, which recovers the standard MVI setting. In this work, we address this gap by presenting a suite of high-order methods that converge to $\ell_p$-norm stationary points for a suitable range of $\rho > 0$, thereby circumventing previous fundamental challenges in $\ell_p$ settings. We further show convergence for high-order smooth \textit{monotone} operators, generalizing Adil et al. (2022) to the case where $p \geq 2$, and we extend our Euclidean techniques to continuous-time settings. 
\end{abstract}


\input{Sections/1.Introduction}
\input{Sections/2.RelatedWorks}
\input{Sections/3.Preliminaries}
\input{Sections/4b._Lp}
\input{Sections/4.MainTheorems}

\input{Sections/7.Conclusion}

\bibliography{references}
\bibliographystyle{plainnat}

\newpage
\begin{appendices}

\section*{Appendix}
\addcontentsline{toc}{section}{Appendix}

\tableofcontents
\clearpage
\input{Sections/Appendix}

\end{appendices}

\end{document}

%% file: math_commands.tex
\usepackage{amsmath,amsfonts,bm}

\def\eqref#1{equation~\ref{#1}}

\def\1{\bm{1}}

\DeclareMathAlphabet{\mathsfit}{\encodingdefault}{\sfdefault}{m}{sl}
\SetMathAlphabet{\mathsfit}{bold}{\encodingdefault}{\sfdefault}{bx}{n}

\newcommand{\R}{\mathbb{R}}

\DeclareMathOperator*{\argmin}{arg\,min}

%% file: macrosarxiv.tex
\newcommand{\innp}[1]{\left\langle #1 \right\rangle}

\newcommand{\cx}{\mathcal{X}}
\newcommand{\cy}{\mathcal{Y}}

\newcommand{\cz}{\mathcal{Z}}

\graphicspath{{imgs/}}

\makeatletter
\def\mathcolor#1#{\@mathcolor{#1}}
\def\@mathcolor#1#2#3{%
  \protect\leavevmode
  \begingroup
    \color#1{#2}#3%
  \endgroup
}
\makeatother

\newcommand*{\vsepfbox}[1]{%
  \begingroup
    \sbox0{\fbox{#1}}%
    \setlength{\fboxrule}{0pt}%
    \mbox{\kern-\fboxsep\fbox{\unhbox0}\kern-\fboxsep}%
  \endgroup
}

\theoremstyle{plain} \numberwithin{equation}{section}

\theoremstyle{definition}
\newtheorem{theorem}{Theorem}[section]
\numberwithin{theorem}{section}
\newtheorem{corollary}[theorem]{Corollary}
\newtheorem{lemma}[theorem]{Lemma}

\theoremstyle{definition}
\newtheorem{definition}[theorem]{Definition}

\newtheorem{example}[theorem]{Example}

\newtheorem{assumption}{Assumption}

\makeatletter
\newcommand{\subalign}[1]{%
  \vcenter{%
    \Let@ \restore@math@cr \default@tag
    \baselineskip\fontdimen10 \scriptfont\tw@
    \advance\baselineskip\fontdimen12 \scriptfont\tw@
    \lineskip\thr@@\fontdimen8 \scriptfont\thr@@
    \lineskiplimit\lineskip
    \ialign{\hfil$\m@th\scriptstyle##$&$\m@th\scriptstyle{}##$\hfil\crcr
      #1\crcr
    }%
  }%
}
\makeatother

\newcommand\calZ{\mathcal{Z}}
\newcommand\hoeg{\textsc{hoeg+}}

\newcommand{\pa}[1]{\left(#1\right)}
\newcommand{\ang}[1]{\left<#1\right>}

\newcommand{\norm}[1]{\ensuremath{\left\lVert #1 \right\rVert}}

%% file: Sections/1.Introduction.tex
\section{Introduction}

Variational inequalities provide a powerful framework for representing important problems across various fields, from economic equilibrium modeling to network flow optimization, reinforcement learning, and adversarial machine learning. The central problem of VIs is to find a point $z^* \in \mathcal{Z}$, where $\mathcal{Z} \subseteq \mathbb{R}^d$ is convex, such that
\begin{equation}\label{eq:VIob}
 \ang{F(z),z^* - z} \leq 0 \quad \forall z \in \mathcal{Z},
\end{equation}
for an operator $F:\cz\rightarrow \R^d$. In the case where $F$ is Lipschitz and monotone, algorithms such as the extragradient method \citep{korpelevich1976extragradient} and its generalizations such as Mirror Prox \citep{nemirovski2004prox} are able to converge  at a rate of $O(1/\epsilon)$ (for the $\epsilon$-approximate objective $\ang{F(z),z^*-z}\leq \epsilon$), and this is known to be tight for first-order methods \citep{ouyang2021lower}. However, as the complexity (and inherent non-convexity) of the underlying models for these large-scale problems increases, so too does the need to expand \emph{beyond} the monotone setting.

Unfortunately, in the most general constrained non-monotone setting, it would appear there is not much hope for efficiently finding a stationary point, as even doing so approximately has been shown to be FNP-complete \citep{daskalakis2021complexity}. Recently, however, there has been significant interest in overcoming these difficulties by looking instead at certain \emph{structured} non-monotone problems. Specifically, previous work by \citet{diakonikolas2021efficient} usefully characterizes problems based on a weakening of the standard Minty condition. To handle problems that satisfy this weak-MVI condition (whose weakness is parameterized by $\rho$), \citet{diakonikolas2021efficient} further generalizes the extragradient method to provide $\epsilon$-approximate small operator norm guarantees (with respect to the Euclidean norm of the operator, i.e., $\|F(z)\|_2 \leq \epsilon$) at a $O(1/\epsilon^2)$ rate, for $\rho \in [0, \frac{1}{4L_1})$ ($L_1$ is the Lipschitz continuity constant of the objective operator). One of the key assumptions in this setting is that of the existence of at least one solution.

However, while \citet{diakonikolas2021efficient} also provide small operator norm guarantees in terms of more general $\ell_p$ norms (for $p \in [1,\infty)$), in the case where $p\neq 2$, their results only hold when $\rho = 0$, which recovers the standard Minty variational inequality setting~\citep{mertikopoulos2018optimistic}. Indeed, as noted by \cite{diakonikolas2021efficient} (see, e.g., Remark 4.3), extending their $\ell_p$-geometry results to this setting for \emph{any} nontrivial range of $\rho$ appears to be a significant challenge, as there are several obstacles in generalizing the results in $\ell_p$-geometry to the weak-MVI setting.

To address this issue, we consider whether any advantage can be gained by additionally considering \emph{$s^{th}$-order smoothness} conditions, along with appropriately suited high-order algorithms. Indeed, we show that for a non-trivial range of $s$ ($1\leq s$) and $p$ ($2 \leq p \leq s+1$), we are able to establish $\epsilon$ approximate $\ell_p$-norm of the operator for a strictly positive range of $\rho$ in a dimension-independent number of iterations. 

\subsection{Our Contributions}
The main contributions of our work are as follows:
\paragraph{Discrete-time $\ell_p$ setting.} 

In the $s^{th}$-order, weak-MVI, $\ell_p$-geometry setting, our algorithm $\ell_p$-\hoeg~achieves a convergence rate of $O(1/\epsilon^\frac{p}{s})$, i.e.,  the output $z_{out}$ of our algorithm satisfies $\|F(z_{out})\|_{p^*} \leq \epsilon$ in $O(1/{\epsilon^{\frac{p}{s}}})$ steps (Theorem \ref{thm:weakmvi}) for $s>1,s+1\geq p\geq 2$ and a strictly positive range of weak-MVI parameter $\rho$, thereby addressing a key challenge observed by \cite{diakonikolas2021efficient}. Furthermore, for the monotone case, we extend the HOMVI algorithm to the $\ell_p$ setting, and provide an $\epsilon$-approximate solution to the VI objective Eq.~\eqref{eq:VIob} at a rate of $O(1/{\epsilon^\frac{p}{s+1}})$ (Theorem \ref{thm:maincomplexity}).

\paragraph{Discrete-time $\ell_2$ setting.}

We extend the \textsc{eg+} algorithm of \cite{diakonikolas2021efficient} to all higher orders $s\geq 2$, thereby achieving $\|F(z_{out})\|_{2} \leq \epsilon$ in $O(1/{\epsilon^{\frac{2}{s}}})$ steps under the $s^{th}$-order weak-MVI setting. In particular, these results hold for a strictly positive range of $\rho$ for all orders $s \geq 2$ (Theorem \ref{theorem:balanced}). This result complements those of the general $\ell_p$-geometry by providing a different range of $\rho$, depending on the appropriate smoothness parameters.

\paragraph{Continuous-time setting.}
 We further consider the continuous-time regime, whereby we rely on the rescaled dynamics of \cite{lin2022continuous}. In this setting, we obtain a continuous-time rate of $O(1/{t^\frac{s}{2}})$ on the operator-norm objective (Theorem \ref{thm:continuoustime}), which provides a natural counterpart to our results for \hoeg~in the discrete time weak-MVI setting.

%% file: Sections/2.RelatedWorks.tex
\subsection{Additional Related Works}
Recently, several works \citep{bullins2022higher, jiang2022generalized,adil2022optimal, lin2024perseus} have shown how to achieve $\epsilon$-approximate weak solutions for monotone variational inequalities under $s^{th}$-order oracle access, with $O(1/\epsilon^\frac{2}{s+1})$ rates \citep{adil2022optimal, lin2024perseus}. However, these works have focused predominantly on the monotone operator setting, though \citet{lin2024perseus} also show a rate of $O(1/\epsilon^\frac{2}{s})$ under the Minty condition. In addition, \citet{lin2022continuous} provide a rate of $O(1/\epsilon^\frac{2}{s})$ to reach a point with small operator norm in the monotone setting. Furthermore, the higher-order works mentioned above all depend on solving a higher-order sub-problem. For convex problems, it has been shown \citep{nesterov2006cubic,nesterov2021implementable} that the sub-problem can be solved for $s=2,3$ in polynomial (linear system solve) time, and \citet{carmon2020acceleration} provide techniques to efficiently solve related sub-problems under local stability assumptions on the Hessian.

The continuous-time regime has proven to be effective in analyzing the performance of algorithms for minimization problems \citep{latz2021analysis,wilson2016lyapunov,wibisono2016variational,shi2021understanding,li2019stochastic}, min-max optimization problems \citep{lin2022continuous,vyas2023competitive,compagnoni2024sdes}, and in the study of continuous games \citep{mazumdar2020gradient}. \citet{latz2021analysis} analyzes stochastic gradient descent in the continuous-time regime while \citet{wilson2016lyapunov,wibisono2016variational,shi2021understanding,malladi2022sdes} provide a continuous-time perspective of the accelerated and adapted variants of gradient descent. \citet{lin2022continuous} study the continuous-time version of the dual-extrapolation algorithm, and \citet{vyas2023competitive} build on the competitive gradient descent algorithm \citep{schafer2019competitive} to design a new min-max algorithm based on re-scaling the cross-terms of the Jacobian of the operator. 

%% file: Sections/3.Preliminaries.tex
\section{Preliminaries}

In this section we discuss the notations and key assumptions that formulate the setting for our algorithm. We start by defining the approximate and exact stationary points of an operator $F$. In the following definition and throughout the paper we define $p^*=\frac{p}{p-1}$ to be the dual of $p \geq 1$.

\begin{definition}[Stationary points] 
A point $z\in \cz\subseteq\mathbb{R}^d$ is an $\epsilon$-approximate stationary point of the operator $F$ if
$$\|F(z)\|_{p^*}\leq \epsilon,$$ and it is an exact stationary point if $\|F(z)\|_{p^*}=0.$
\end{definition}
We next define the solution set to the Stampacchia variational inequality for any operator $F$. 
\begin{definition}[Solution set 
$\cz^*$]
We refer to the set of solutions of the Stampacchia Variational Inequality (SVI) as the set $\cz^* \subseteq \cz $:
$$\cz^* = \{z^* : \langle F(z^*),z-z^*\rangle\geq 0 ~\forall z\}. $$
\end{definition}
 We assume $\cz^*\neq \emptyset$. We proceed by defining monotonicity of an operator.

\begin{definition}[Monotonicity]\label{def:monotonicity}
An operator is \emph{monotone} if for any two points $z_a,z_b\in \cz$ we have
$$\langle F(z_a)-F(z_b),z_a-z_b \rangle \geq 0.$$
\end{definition}

Standard examples of monotone operators include the gradient $\nabla f$ of a convex function $f(x)$ and the concatenated gradient $(\nabla_x f, -\nabla_y f)$ for a convex-concave function $f(x,y)$.

We now present the definition of comonotonicity, which generalizes monotonicity and which provides a key non-monotone condition for first-order algorithms~\citep{lee2021fast}.

\begin{definition}[$\rho$-comonotone]
An operator $F$ is $\rho$-comonotone for some $\rho \in \R$ if for any two points $z_a,z_b\in \cz$ we have
$$\langle F(z_a)-F(z_b),z_a-z_b \rangle > \rho \|F(z_a)-F(z_b)\|_{p^*}^2$$
\end{definition}
Note that $\rho$-comonotonicity implies monotonicity for $\rho\geq 0$. Inspired by the weak-MVI condition in \citet{diakonikolas2021efficient}, we generalize the weak-MVI condition to $\ell_p$-geometry for an $s^{th}$-order condition, which will be a key assumption underlying the convergence guarantees of our algorithm.

\begin{assumption}[$s^{th}$-Order weak-MVI]\label{assump:pwmvi}
There exists $z^* \in \cz^*$ such that:
\begin{equation}\label{assmpt:balanced}\tag{\textsc{a}$_1$}
    (\forall z \in \R^d):\quad 
    \innp{F(z), z - z^*} \geq -\frac{\rho}{2} \|F(z)\|_{p^*}^{\frac{s+1}{s}},
\end{equation}
for some parameter $\rho>0$. For $\rho=0$, the point $z^*$ such that the above condition holds $\forall~z\in \mathcal{Z}$ is called an MVI solution.
\end{assumption}

For $\rho=0$ the condition is the well-known MVI condition \citep{mertikopoulos2018optimistic}. Furthermore, for $s=1$, $\rho$-comonotonicity implies $-\frac{\rho}{2}$ weak-MVI. Overall, we have $\text{monotonicity} \;\Rightarrow\; \text{MVI} \;\Rightarrow\; \text{weak-MVI}$.

In order to define higher-order algorithms, we require the following series of definitions involving higher-order gradients. We begin by defining the directional derivatives. 
\begin{definition}[Directional derivative]
    Consider a $k$-times differentiable operator $F:\mathcal{Z} \rightarrow \mathbb{R}^d$, and let $z,h \in \mathcal{Z}$, $\mathcal{Z}\subseteq \mathbb{R}^d$. For $r\leq k+1$, we let
    $$\nabla^k F(z)[h]^r = \frac{\delta^k
    }{\delta h}|_{t_1=0,\dots t_r=0}F(z+t_1h+\dots +t_rh)$$
    denote the $k^{th}$ directional derivative of $F$ at $z$ along $h$.
\end{definition}
We now move on to the definition of the Taylor expansion of an operator $F$.
\begin{definition}
We define the $s^{th}$-order Taylor approximation of $F$ centered at $z_a$ to be
\begin{equation}\label{eq:tau-def}
    \mathcal{T}_s(z_b;z_a) := \sum_{i=0}^s \frac{1}{i!}\nabla^i F(z_a)[z_b-z_a]^i.
\end{equation}
\end{definition}

Next, we define the higher-order smoothness of an operator.
\begin{definition}
    [$s^{th}$-Order Smoothness in the $\ell_p$-norm]

    An operator $F$ is $s^{th}$-order smooth in the  $\ell_p$-norm with Lipschitz parameter $L_{s,p}$ if for any two points $z_a,z_b\in \cz$ we have
\begin{equation}\label{assmpt:smooth}\tag{\textsc{a}$_2$}
    \|F(z_b)-\mathcal{T}_{s-1}(z_b;z_a)\|_{p^*}\leq \frac{L_{s,p}}{s!}\|z_b-z_a\|_{p}^s.
\end{equation}

\end{definition}

\begin{definition}
Let $s,p>0$ and $\nu\in \mathbb{R}$, for any two points $z_a,z_b \in \cz$. We define $\Phi_{s,p}^\nu(z_a,z_b)$ as the regularized Taylor approximation 
\begin{align}\label{eq:phi-def}
    &\Phi_{s,p}^\nu(z_b;z_a) :=   \mathcal{T}_{s-1}(z_b;z_a) +\frac{2^\nu L_{s,p}}{s!}\|z_b-z_a\|_p^{s-1}(z_b-z_a).
\end{align}
\end{definition}

We use $(s-1)$-order Taylor expansions since the remainder is controlled by the
$s^{th}$ derivative.
\begin{definition}
We define the Bregman divergence with respect to a differentiable function $h:\mathbb{R}^d\rightarrow \mathbb{R}$ for any two points $z_a,z_b \in \cz$ as 
    $$\omega_h(z_a,z_b) = h(z_a)-h(z_b) -\ang{\nabla h(z_b),z_a-z_b}.$$
\end{definition}

\begin{definition}
For any two points $z_a,z_b\in \cz$, we define $\Psi_{s,p}^\nu(z_a,z_b)$ as the Bregman regularized Taylor approximation
\begin{align}\label{eq:psi-def}
    &\Psi_{s,p}^\nu(z_b;z_a) :=   \mathcal{T}_{s-1}(z_b;z_a) +\frac{2^\nu L_{s,p}}{s!}\omega_h(z_b,z_a)^\frac{s-1}{2}(h(z_b)-h(z_a)),
\end{align}
where $h(z)=\|z\|_p^p$ for some parameter $p\geq 2$.
\end{definition}

%% file: Sections/4b._Lp.tex
\section{Discrete Time $\ell_p$-Geometry}

In this section, we develop our results under $\ell_p$ geometry for both the monotone and weak-MVI settings. A key ingredient in our analysis is a structural property of the Bregman divergence induced by Theorem~\ref{lem:omegapnorm}. 

\subsection{Monotone Operators}

As a warm-up, we begin with the monotone setting considered in \cite{adil2022optimal}, whereby we provide a natural extension of their techniques to $\ell_p$-geometries using an algorithm that utilizes both high-order and $\ell_p$-based mirror-descent update steps. In particular, the following theorem shows how our algorithm $\ell_p$-HOMVI (Algorithm \ref{alg:mainalg1}) generalizes previous results.

\begin{table}[h]
\centering
\begin{tabular}{l c c}
\toprule
 & $s\geq 1,p=2$ & $s\geq 1, p\geq2$ \\
\midrule
\textsc{homvi} \citep{adil2022optimal}
& $O({1}/{\epsilon^{\frac{2}{s+1}}})$ 
& -- \\
\textbf{$\ell_p$-\textsc{homvi} (Ours)} 
& $O(1/\epsilon^{\tfrac{2}{s+1}})$ 
& $O(1/\epsilon^{\tfrac{p}{s+1}})$ \\
\bottomrule
\end{tabular}
\caption{Rates of convergence for monotone operator $F$.}
\label{table:monotone}
\end{table}

\begin{restatable}{theorem}{MainComplexity}\label{thm:maincomplexity}
For $s\geq 1,~p \geq 2$, let $F$ be an operator that is $L_{s,p}$ $s^{th}$-order smooth and monotone. Upon running the $s^{\text{th}}$ order instance of the $\ell_p$-HOMVI (Algorithm \ref{alg:mainalg1}) on  $F$ we have that the output $z_{out}$ satisfies $\langle F(z), z_{out} - z \rangle \leq \varepsilon$ for all $z \in \mathcal{Z}$
in at most $2^{p+1}(\frac{s^s}{p^p})^\frac{1}{s+1}
\left(\frac{1}{s+1-p}\right)^{\frac{s+1-p}{s+1}} D\left( \frac{L_{s,p} }{s!\epsilon} \right)^{\frac{p}{s+1}}$ iterations, where $D=\max_z \omega(z,z_0) \in \cz$.
\end{restatable}
\begin{proof}(Sketch)
    The proof generally follows by replacing the potential function $h(z) = \|z\|_2^2$ with the $\ell_p$-based potential function $h(z) = \|z\|_p^p$ in the proof of analysis of Theorem 3.2 and Theorem 3.3 in \cite{adil2022optimal}. The full proof is present in Appendix \ref{app:lphomvi}.
\end{proof}

 $\ell_p$-HOMVI performs two updates in the spirit of Mirror Prox~\citep{nemirovski2004prox} and HOMVI~\citep{adil2022optimal}, though with both steps now defined in terms of $\norm{\cdot}_p$. Similar to the HOMVI algorithm, the learning rate used for the first-order update is based on the change due to the higher-order update. The parameter $\nu$ controls the regularization in $\Psi_{s,p}^\nu$ used for the higher-order step, and we choose $\nu$ so that it maximizes the rate of convergence.

\begin{algorithm}
\caption{$\ell_p$-HOMVI}\label{alg:mainalg1}
\DontPrintSemicolon

\KwIn{$s \geq 1,\; p \geq 2$}

\textbf{Initialize:} 
$z_0 \in \cx \times \cy$, 
$h(z) \gets \|z\|_p^p$, 
$\nu \gets \frac{(s+1)(p+1)}{p}-\log_2 p+\frac{p-1}{p}\log_2 s$\;

\For{$k = 0$ \KwTo $K$}{

$z_{k+\tfrac{1}{2}}$ s.t.
$\langle \Psi_{s,p}^\nu(z_{k+\frac{1}{2}},z_k),
z_{k+\frac{1}{2}}-z\rangle \leq 0,
\ \forall z\in \mathcal{Z}$\;

$\lambda_k =
\frac{1}{2^\nu}
\,\omega_h(z_{k+\tfrac{1}{2}}, z_k)^{-\tfrac{s+1-p}{p}}$\;

$z_{k+1} =
\argmin_{z'' \in \mathbb{R}^d}
\Big\{
\innp{F(z_{k+\tfrac{1}{2}}), z'' - z_{k+\tfrac{1}{2}}}
+ \frac{L_{s,p}}{s! \lambda_k} \omega_{\|z\|_p^p}(z'', z_k)
\Big\}$\;
}

\Return
$z_{\text{out}}
= \frac{\sum_{k=1}^K \lambda_k z_k}
{\sum_{k=1}^K \lambda_k}$\;

\end{algorithm}

\subsection{Weak-MVI}

In this sub-section we present a key theorem dealing with the most general setting of our paper. We begin with Theorem \ref{lem:omegapnorm} that plays an important part in the proof of Theorem \ref{thm:maincomplexity}. 

\begin{restatable}{theorem}{omegapnorm}\label{lem:omegapnorm}
Consider the potential function $h(z) = \|z\|_p^s$. For all $z_a, z_b \in \mathcal{Z}$, we have
$\omega_{h}(z_a,z_b) \geq \frac{4}{2^s} \|z_a-z_b\|_p^s$ for $s\geq p \geq 2$.
\end{restatable}
\begin{proof}(Sketch)
    The proof follows by showing that the minimum of the ratio $f(z_a,z_b)=\frac{\omega_h(z_a,z_b)}{\|z_a-z_b\|_p^s}$ for $h=\|z\|_p^s$ occurs at $z_a=-z_b\neq 0$. This is shown by reparameterizing $f$ with parameters $q=(z_a-z_b)/2$ and $r=(z_a+z_b)/2$ and then showing that the minimum occurs at $q\neq 0,r=0$. The complete proof is provided in Appendix \ref{app:lphomvi}.
\end{proof}

Theorem \ref{lem:omegapnorm} relates the Bregman divergence $\omega_h(z_a,z_b)$ to $h(z_a,z_b)$ where $h(z_a,z_b) = \|z_a-z_b\|_p^s$ with $s\geq p\geq 2$. In doing so it extends extends Lemma 4 in \citet{nesterov2008accelerating} to the case of general $\ell_p$ norms (for $p \geq 2$).

\begin{algorithm}  \caption{$\ell_p$-\textsc{hoeg+}}\label{alg:mainalg2}
  \DontPrintSemicolon
  \KwIn{$s \geq 1, s +1\geq p\geq 2$}
  \textbf{Initialize:}
    $z_0 \in \cz,$\;
    $\nu \leftarrow \frac{s}{s+1}\pa{\log_2\pa{\frac{s(s+2)2^{1-\frac{1}{s}}}{(s+1)^{1+\frac{1}{s}}}}}$\\
  \For{$k = 0$ \KwTo $K$}{
      $z_{k+\frac{1}{2}}~\textrm{s.t.}~\mathcal{T}_{s-1}(z_{k+\frac{1}{2}}; z_k) 
      + \nabla_u \Big(\frac{2^\nu L_{s,p}}{s!}\|u\|_p^{s+1}\Big)|_{u = z_{k+\frac{1}{2}} - z_k} = 0$\;

      $\lambda_k = (\frac{1}{2^{p+1}})^\frac{s}{s+1} \left( \frac{s+2}{p} \right)^{- \frac{s+1 - p}{s+1}}
      \,\|z_{k+\frac{1}{2}} - z_k\|_p^{- (s+1 - p)}$\;

      $z_{k+1} = \argmin_{z'' \in \mathbb{R}^d} 
      \Big\{\;
        \innp{F(z_{k+\frac{1}{2}}), z'' - z_{k+\frac{1}{2}}} 
        + \frac{L_{s,p}}{s! \, \lambda_k} \, \omega_{\|z\|_p^p}(z'', z_k) 
      \;\Big\}$\;
  }
  \Return $z_{\text{out}} = \argmin_{0 \le k \le K} 
  \left\| F\left(z_{k+\frac{1}{2}} \right) \right\|_{p^*}$\;
\end{algorithm}

\begin{table}[h]
\centering
\begin{tabular}{l c c}
\toprule
 & {\small \textit{weak}-MVI $(s=1)$} 
 & {\small \textit{weak}-MVI $(s+1 \geq p\geq 2)$} \\
\midrule
\textsc{eg+} \citep{diakonikolas2021efficient}
& $O(1/\epsilon^p)$ 
& -- \\
\textbf{$\ell_p$-\textsc{hoeg+} (Ours)} 
& $O(1/\epsilon^{p})$ 
& $O(1/\epsilon^{\tfrac{p}{s}})$ \\
\bottomrule
\end{tabular}
\caption{Rates of convergence for operator $F$ satisfying \textit{weak}-MVI}
\label{table:weak-mvi}
\end{table}

  In order to address the weak-MVI setting, we present $\ell_p$-\textsc{hoeg+} (Algorithm \ref{alg:mainalg2}), which uses a higher-order steepest descent update based on the Taylor model with the regularization term proportional to $\|u\|_p^{s+1}$ and a mirror descent update based on the potential function $h(z) = \|z\|_p^p$. Since $\ell_p$-\textsc{hoeg+} is designed for the unconstrained setting, the higher-order step involves an equality rather than a VI subproblem as in $\ell_p$-HOMVI. Additionally, as opposed to choosing $\nu$ to optimize the constant in the convergence rate as in the case of Algorithm \ref{alg:mainalg1}, we choose it to maximize the range of $\rho$ in the \textit{weak}-MVI condition of the objective operator for which the algorithm converges. The following theorem establishes the rates of convergence for our algorithm in the weak-MVI setting.

\begin{restatable}{theorem}{WeakMVI}\label{thm:weakmvi}
\small
 For $s\geq1$, $s+1 \geq p \geq 2$, let $F$ be an $L_{s,p}$ $s$-th order smooth operator w.r.t. $\|\cdot\|_p$ satisfying \ref{assmpt:balanced}, the $s$-th order weak-MVI condition with parameter 
 \[
0 <\rho \;<\; 2\left(2^\nu-\frac{s}{2^{\frac{\nu-s+1}{s}} (s+1)^{1+\frac{1}{s}}}\right)
\left( \frac{s!}{((s+1)2^\nu+1)^{\,s+1}\,L_{s,p}} \right)^{\frac{1}{s}}\ \ .
\]
 Upon running the $s^{th}$-order instance of Algorithm~\ref{alg:mainalg2} on operator $F$ with $z_0 \in \R^d$, the best iterate satisfies
\begin{align*}
\min_{0 \leq k \leq K} 
\|F(z_{k+\tfrac{1}{2}})\|_{p^*}^{\tfrac{p}{s}}
&\;\leq\; 
\frac{1}{K}\sum_{k=1}^K 
\|F(z_{k+\tfrac{1}{2}})\|_{p^*}^{\tfrac{p}{s}}
\;\leq\; 
\frac{\omega_p(z^*,z_0)}{KC_{s,p} c_{s,p,\rho}}\, ,
\end{align*}
 where $z^* \in \cz^*$, 
 $c_{s,p,\rho} = \left[
\left( s(s+1) - \frac{{(4s)}^{\frac{s-1}{s}}}{(s+1)^{1+\frac{2}{s}}} \right)
\left( \frac{s!}{(s+1)^2L_{s,p}} \right)^{\frac{s+1}{s}} - \frac{s! \rho}{2L_{s,p}}
\right]$, $C_{s,p} = \frac{1}{2^{\frac{(p+1)(s+1)+\nu(p-1)}{s}}} 
\left(
\frac{s!}{pL_{s,p}}
\right)^{\frac{p-(s+1)}{s}}$ and $\omega_p$ represents $\omega_{\|z\|_p^p}(\cdot,\cdot)$.
As a consequence, for $\epsilon > 0$, the best iterate satisfies
\[
\min_{0 \leq k \leq K} 
\|F(z_{k+\tfrac{1}{2}})\|_{p^*}
\;\leq\; \epsilon
\]
after $K = O\!\bigl(\tfrac{1}{\epsilon^{p/s}}\bigr)$ iterations.
\end{restatable}

\begin{proof}(Sketch)
    The proof, which builds on the techniques of \cite{diakonikolas2021efficient}, provides for extending to a higher-order analysis. As a consequence, this approach lets us to use the 'half-step trick' (an extension of Lemma \ref{lemma:upper-b}) for the general $\ell_p$ case. The full proof is provided in Appendix \ref{app:lphomvi}.
\end{proof}

Note that in contrast to Algorithm \ref{alg:mainalg1}, the range of $p$ for which convergence is shown for Algorithm \ref{alg:mainalg2} is upper bounded by $s+1$, which is due to the use of Theorem \ref{lem:omegapnorm} in the proof. We find that $\nu\approx \frac{s}{s+1}\pa{\log_2\pa{\frac{s(s+2)2^{1-\frac{1}{s}}}{(s+1)^{1+\frac{1}{s}}}}}$ is the approximate maxima of the upper bound of the range of $\rho$. We note that for all $s$, the upper-bound on $\rho$ (after choosing $\nu$ appropriately) is positive. This can be seen as the sign of the upper-bound only depends on the first factor in the expression, which can be positive for large enough $\nu$. This result suggests that the challenge of obtaining convergence via first-order methods for the weak-MVI setting in $\ell_p$-geometry (with $\rho>0$) can be overcome (for higher-order smooth functions) by using higher-order methods. Previous work by~\cite{diakonikolas2021efficient} obtains a rate of $\frac{1}{\epsilon^p}$ using a first order method (\textsc{eg+}) for the MVI-setting ($\rho=0$), but acknowledge the difficulty of moving beyond the $\rho = 0$ regime.

%% file: Sections/4.MainTheorems.tex
\section{Discrete-Time $\ell_2$-Geometry}
We now present our algorithm and results for the special case of $\ell_2$ geometry. We discuss the results in the balanced weak-MVI (\ref{assmpt:balanced}) and imbalanced weak-MVI (\ref{assmpt:imbalanced}) settings, presenting the latter in the Appendix \ref{app:qth}. The imbalanced setting refers to the scenario when the the exponent on $\|F\|$ is not $\frac{s+1}{s}$ for the $s^{th}$ order method but rather a decoupled parameter $q$. Algorithm \ref{alg:mainalg3}, \textsc{hoeg+} is designed for convergence in the $\ell_2$ geometry and uses a higher order step just like $\ell_p$ \textsc{hoeg+}. However, we note that the higher-order step in $\ell_p$ \textsc{hoeg+} involves a term proportional to $\nabla_z\|z\|_p$ while the higher order step in \textsc{hoeg+} does not. The first order update in \textsc{hoeg+} is a steepest descent step in the $\ell_2$-norm as opposed to a mirror descent step in the dual space induced by $h(z) = \|z\|_p^p$ for $\ell_p$ HOMVI and $\ell_p$ \textsc{hoeg+}. 

\begin{table*}[h]
\centering
\caption{Rates of convergence for \textit{weak}-MVI in the $\ell_2$ geometry.}
\label{tab:l2geo}
\begin{tabular}{c cc cc}
\toprule
& \multicolumn{2}{c}{\small $s=1,\;p=2$} 
& \multicolumn{2}{c}{\small $s>1,\;p=2$} \\ 
\cmidrule(lr){2-3} \cmidrule(lr){4-5}
& {\small $\rho = 0$} 
& {\small $\rho \in \big[0, \min\{\tfrac{1}{4L_1}, \tfrac{0.21}{L_1^2}\}\big]$} 
& {\small $\rho = 0$} 
& {\small $\rho \leq (s+\frac{1}{2})(\tfrac{s!}{(s+\frac{3}{2})L_{s,2}})^\frac{s+1}{s}$}\\ 
\midrule
\cite{diakonikolas2021efficient} 
  & $\tfrac{1}{\epsilon^2}$ & $\tfrac{1}{\epsilon^2}$ 
  & -- & -- \\
\midrule
\cite{lin2024perseus} 
  & $\tfrac{1}{\epsilon^2}$ & -- 
  & $\tfrac{1}{\epsilon^{2/s}}$ & -- \\
\midrule
\textbf{Our work} (\textsc{hoeg+})
  & $\tfrac{1}{\epsilon^2}$ & $\tfrac{1}{\epsilon^{2}}$ 
  & $\tfrac{1}{\epsilon^{2/s}}$ & $\tfrac{1}{\epsilon^{2/s}}$ \\
\bottomrule
\end{tabular}
\end{table*}

\begin{algorithm}[H]
  \KwIn{$s\geq 1$}
  \textbf{Initialize:}  $z_0 \in \cz$, $\nu\leftarrow \begin{cases}
0.656      & \text{if } s = 1, \\[4pt]
\log_2(s+0.5+\frac{1}{5s}-\frac{1}{4s^{3}})       & \text{if } s \geq 2.
\end{cases}$ \\
  \For{k = 0 to k = K}{
        $z_{k+\frac{1}{2}} ~~\textrm{s.t.} ~~\Phi_{s,2}^l(z_{k+\frac{1}{2}},z_k)=0$\\
        
    $\lambda_k = \frac{1}{2^\nu}\|z_{k+\frac{1}{2}} - z_k\|^{1-s}$
    \\
    $z_{k+1} = \argmin_{z'' \in \mathbb{R}^d} \Big\{ \innp{F(z_{k+\frac{1}{2}}), z'' - z_{k+\frac{1}{2}}} + \frac{L_{s,2}}{s!\lambda_k}\|z_{k}'' - z_k\|^2\Big\}$\\
  }
       \Return $z_{out} = \argmin_{\{z_{k+\frac{1}{2}}\}_{k=0}^K}\|F(z_{k+\frac{1}{2}})\|$ \caption{\textsc{hoeg+}}\label{alg:mainalg3}
\end{algorithm}

Before we present our main results for the $\ell_2$ geometry we present a key lemma crucial to the small operator norm weak-MVI setting (a similar result holds for the general case $p>2$, see proof of Theorem \ref{thm:weakmvi} in Appendix \ref{app:lphomvi}).

\begin{restatable}{lemma}{UpperBoundLemma}\label{lemma:upper-b}
For $z_{k+\tfrac{1}{2}}, z_k$ obtained from an $s^{th}$-order instance of \hoeg~(Algorithm~\ref{alg:mainalg3}), we have
\begin{equation}\label{upper-b}
    \|F(z_{k+\tfrac{1}{2}})\|_2 \;\leq\; \frac{(2^\nu+1)L_{s,2}}{s!}\,\|z_{k+\tfrac{1}{2}} - z_k\|_2^s.
\end{equation}
\end{restatable}

We now prove convergence for our \hoeg~algorithm with a $s^{th}$-order weak-MVI condition where $s$ is closely tied to the order of the instance of \hoeg~used.
\begin{restatable}{theorem}{Balanced}\label{theorem:balanced}
For any $s^{\text{th}}$-order smooth operator $F$ satisfying~\ref{assmpt:balanced} with $p=2$ and
\[
    0< \rho < \left(2^\nu - 2^{-(\nu+2)}\right)
    \left(\tfrac{s!}{(2^\nu+1)L_{s,2}}\right)^{\tfrac{s+1}{s}},
\]
for $\nu$ as defined in Algorithm~\ref{alg:mainalg3}, running \hoeg ~yields iterates $\{z_{k+\frac{1}{2}}\}_{k=0}^K$ such that, for all $K \ge 1$,
\[
    \frac{1}{K+1}\sum_{k=0}^K 
    \|F(z_{k+\tfrac{1}{2}})\|_2^{\tfrac{2}{s}}
    \;\le\;
    \frac{\|z_0 - z^*\|_2^2}{c_2 (K+1)},
\]

where $c_2 = \left(\frac{s!}{(2^\nu+1)L_{s,2}}\right)^{\tfrac{2}{s}}$. In particular the best iterate satisfies,
\[
    \min_{0 \le k \le K}
    \|F(z_{k+\tfrac{1}{2}})\|_2^2
    \le 
    \frac{C}{c_1^s c_2^s (K+1)^s},
\]
where
\[
    C = \|z_0 - z^*\|_2^{2s}, \qquad
    c_1 = 
    \left(
        1-\frac{1}{2^{2\nu+2}}
        - \frac{\rho}{2^\nu}
          \left(\frac{(2^\nu+1)L_{s,2}}{s!}\right)^{\frac{s+1}{s}}
    \right)
\]
\end{restatable}

\begin{proof}
Let $z^*$ be an SVI solution that satisfies the weak-MVI condition.  
Setting $z=z^*$ in Lemma~\ref{lemma:supportive} (\citep{adil2022optimal}) for 
Algorithm~\ref{alg:mainalg3}, we obtain
\begin{align}\label{eq:main4}
    \frac{s!}{L_{s,2}}\sum_{k=0}^K \lambda_k 
    \big\langle F(z_{k+\tfrac{1}{2}}),\,
    z_{k+\tfrac{1}{2}} - z^* \big\rangle
    \le 
    \omega(z^*, z_0)
    - \Bigl(1-\tfrac{1}{2^{2\nu+2}}\Bigr)
      \sum_{k=0}^K (2\lambda_k)^{-\tfrac{2}{s-1}} .
\end{align}

Adding $\sum_{k=1}^K \rho \lambda_k \|F(z_{k+\tfrac{1}{2}})\|_2^{\tfrac{s+1}{s}}$ to both sides of~\eqref{eq:main4}.  
Using the result of Lemma~\ref{lemma:upper-b},
\[
    \rho \lambda_k 
    \|F(z_{k+\tfrac{1}{2}})\|_2^{\tfrac{s+1}{s}}
    =
    \frac{\rho \|F(z_{k+\tfrac{1}{2}})\|_2^{\tfrac{s+1}{s}}}
         {2^\nu \|z_k - z_{k+\tfrac{1}{2}}\|_2^{s-1}}
    \le
    \frac{\rho}{2^\nu}
    \left(\frac{(2^\nu+1)L_{s,2}}{s!}\right)^{\tfrac{s+1}{s}}
    \|z_k - z_{k+\tfrac{1}{2}}\|_2^2 ,
\]
and therefore
\begin{align}\label{main21}
    \sum_{k=0}^K \lambda_k \Bigl(
        \frac{s!}{L_{s,2}}
        \langle F(z_{k+\tfrac{1}{2}}),\,z_{k+\tfrac{1}{2}} - z^* \rangle
        &+ \rho\,
          \|F(z_{k+\tfrac{1}{2}})\|_2^{\tfrac{s+1}{s}}
    \Bigr)
    \le
    \|z_0 - z^*\|_2^2 \\
    & -
    \Biggl(
        1-\frac{1}{2^{2\nu+2}}
        - \frac{\rho}{2^\nu}
          \left(\frac{(2^\nu+1)L_{s,2}}{s!}\right)^{\tfrac{s+1}{s}}
    \Biggr)
    \sum_{k=0}^K 
    \|z_k - z_{k+\tfrac{1}{2}}\|_2^2 . \nonumber
\end{align}

Set
$c_1 
    =
    1-\frac{1}{2^{2\nu+2}}
    - \frac{\rho}{2^\nu}
      \left(\frac{(2^\nu+1)L_{s,2}}{s!}\right)^{\tfrac{s+1}{s}} .$ Since 
$\rho < 
    \left(2^\nu-2^{-(\nu+2)}\right)
    \left(\tfrac{s!}{(2^\nu+1)L_{s,2}}\right)^{\tfrac{s+1}{s}},$
we have $c_1 > 0$, and the LHS of~\eqref{main21} is nonnegative. Define
$c_2 = \left(\frac{s!}{(2^\nu+1)L_{s,2}}\right)^{\tfrac{2}{s}} .$ From Lemma~\ref{lemma:upper-b},
\[
    c_2\,
    \|F(z_{k+\tfrac{1}{2}})\|_2^{\tfrac{2}{s}}
    \le
    \|z_k - z_{k+\tfrac{1}{2}}\|_2^2 .
\]

Thus, for the half-step iterates of Algorithm~\ref{alg:mainalg3},
\begin{align}
    c_2 \sum_{k=0}^K 
        \|F(z_{k+\tfrac{1}{2}})\|_2^{\tfrac{2}{s}}
    \le
    \sum_{k=0}^K 
        \|z_k - z_{k+\tfrac{1}{2}}\|_2^2
    \le
    \frac{\|z_0 - z^*\|_2^2}{c_1},
\end{align}
which proves the theorem.
\end{proof}

Here, $\nu$ is chosen in order to maximize the range of $\rho$ for the weak-MVI condition while ensuring convergence. We find that the best $\nu$ for $s=1$ is $0.656$ while for $s>1$ follows the expression $\log_2 (s+0.5+\frac{1}{5s}-\frac{1}{4s^3})$ approximately. Note that it is possible to obtain a convergence rate for the $s=2$ {weak}-MVI setting from Theorem \ref{thm:weakmvi}, however we find that the range of $\rho$ obtained by \hoeg~(Algorithm \ref{alg:mainalg3}) can be larger than that obtained by Theorem \ref{thm:weakmvi} for small $L_{s,2}$ owing to the difference in the factor containing $L_{s,2}$. A factor of $L_{s,p}^{-\frac{1}{s}}$ appears in the range of $\rho$ in Theorem \ref{thm:weakmvi}, versus a factor of  $L_{s,2}^{-\frac{s+1}{s}}$ in Theorem \ref{theorem:balanced}.

\input{Sections/5.Experiments}

\section{Continuous-Time $\ell_2$-Geometry}

We now discuss the convergence of the dynamics of a system of differential equations which represents the continuous-time dynamics of the $s^{th}$-order dual extrapolation algorithm \citep{nesterov2007dual}, as given by \citet{lin2022continuous}. The rescaled dynamical system is as follows:
\begin{equation}\label{sys:DE}\tag{Re-DS}
\begin{array}{lll}
  \dot{u}(t) = - \tfrac{F(z(t))}{\|F(z(t))\|^{1-1/s}}, \quad\quad v(t) = z_0 + u(t), &\\
\quad\quad z(t) - v(t) + \tfrac{F(z(t))}{\|F(z(t))\|^{1-1/s}} = \textbf{0}.
\end{array} 
\end{equation}

While \citet{lin2022continuous} show that \eqref{sys:DE} is the dynamical system for dual extrapolation, we know that dual extrapolation is equivalent to the extra-gradient method for our unconstrained Euclidean setting \citep{korpelevich1983extrapolation}. We thus follow \citet{wibisono2016variational} and \citet{lin2022continuous} and choose the above dynamics with re-scaled operator to capture the continuous time analogue of \hoeg.

The following theorem presents our main result for the continuous-time regime. While in the monotone and MVI setting the norm is decreasing monotonically, this is not guaranteed in the weak-MVI setting. A supplementary result for the co-monotonicity condition, under which the dynamics of \ref{sys:DE} results in continuous decrease of the norm $\|F(z_t)\|$, is presented in Corollary \ref{cor:comonotone}.

\begin{restatable}{theorem}{Continuoustime}\label{thm:continuoustime}
Let the operator $F$ satisfy~\eqref{assmpt:balanced} with $\rho < 2$. Then for the continuous-time dynamics~\eqref{sys:DE} we have
\[
\min_{0 \le r \le t} \|F(z(r))\| \le O\!\big(t^{-s/2}\big).
\]
\end{restatable}
For the weak-MVI condition, $\rho<2$ is sufficient to obtain the desired $O(1/{t^\frac{s}{2}})$ rate on the best norm over the trajectory $\min_{0 \leq s \leq t} \|F(s)\|$. We further show that for an operator that is comonotone with $\rho\geq-1$, the dynamics \eqref{sys:DE} are such that $\|F(z(t))\|$ is decreasing and since comonotonicity with $\rho>-1$ implies \footnote{Note that for $\rho=-1$ we still have $\|F(z(t)\|$ decreasing with time but the conditions of \ref{thm:continuoustime} are not satisfied.} weak-MVI with $\rho>2$ we have the conditions of Theorem \ref{thm:continuoustime} satisfied and thus $\|F(z(t))
\| = \min_{0\leq r \leq t}\|F(z(r))\|^2 \leq O(1/{t})$.

\begin{restatable}{corollary}{Comonotone}\label{cor:comonotone}
If the operator $F$ is $\rho$-comonotone with $\rho > -1$, then for the continuous-time dynamics~\eqref{sys:DE} with $s=1$, we have that $\|F(z(t))\|$ is decreasing and
\[
\|F(z(t))\|^2 \;\leq\; O\!\left(\tfrac{1}{t}\right).
\]
\end{restatable}
We provide the proof of this corollary in Appendix \ref{cor:comonotone}.

%% file: Sections/5.Experiments.tex
\subsection{Experiments}\label{sec:experiments}

 We start by briefly discussing the min-max optimization problem that is a special case of the variational inequality problem with the appropriate operators. It is formulated as
\begin{equation}\label{eqn:minmax}
     \min_{x\in \cx} \max_{y \in \cy} f(x,y)
\end{equation}
where $f:\cx\times\cy\rightarrow \mathbb{R}$ may be non-convex in $x$ and non-concave in $y$, and where we assume $f$ is smooth (up to various orders) in both $x$ and $y$. To solve the problem in~\eqref{eqn:minmax} we consider the operator $F=(\nabla_x f,-\nabla_y f)$, derived from the function $f$. For the unconstrained setting, i.e, $\cz=\mathbb{R}^d$, we have $\|F(z^*)\|=0~\forall z^* \in \cz^*$. Furthermore all stationary points of \eqref{eqn:minmax} satisfy the SVI.

\begin{figure*}[htbp]
    \centering
    \begin{minipage}{0.4\textwidth}
        \centering
        \includegraphics[width=\linewidth]{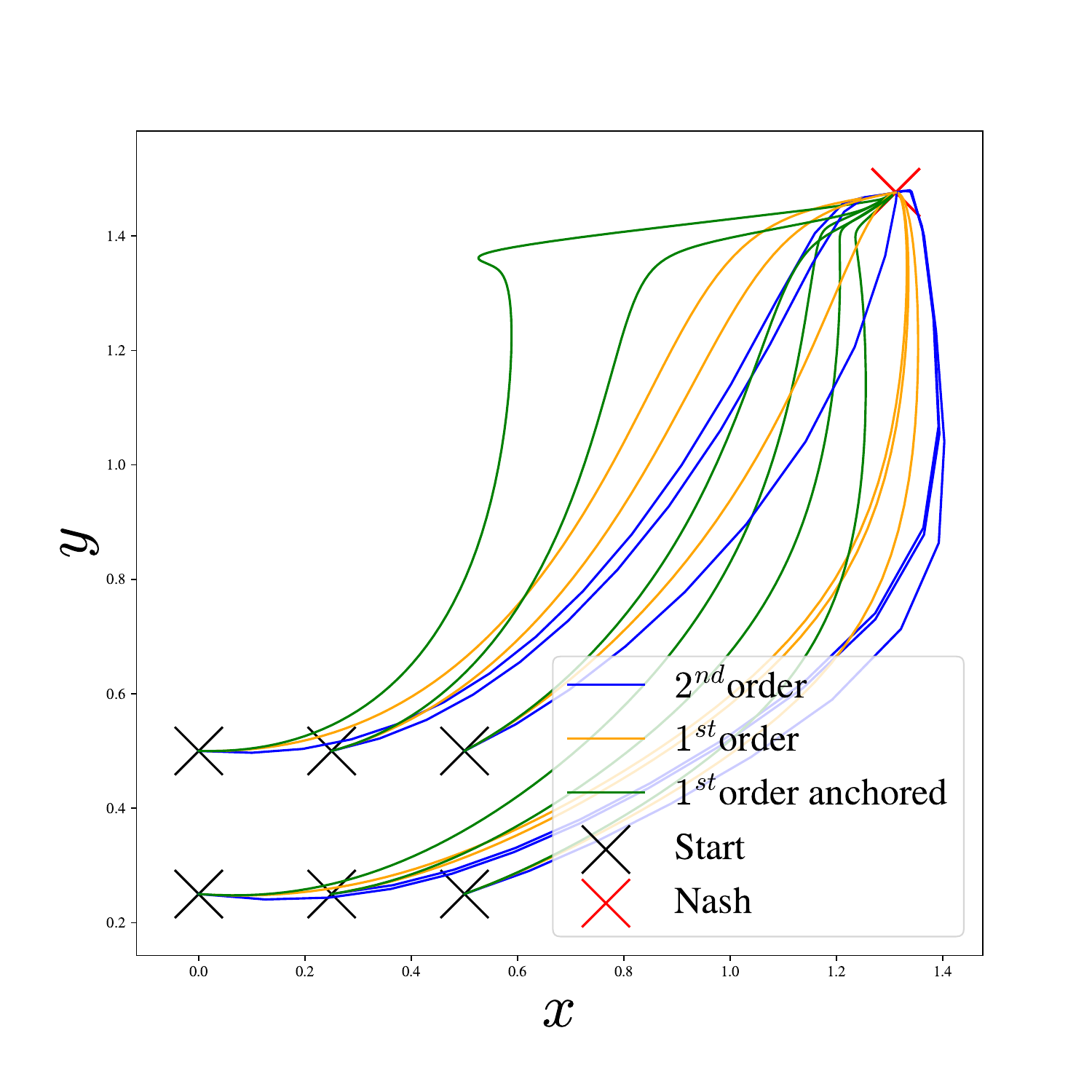}
        \\ (a) The iterates of the algorithm
    \end{minipage}%
    \hfill
    \begin{minipage}{0.4\textwidth}
        \centering
        \includegraphics[width=\linewidth]{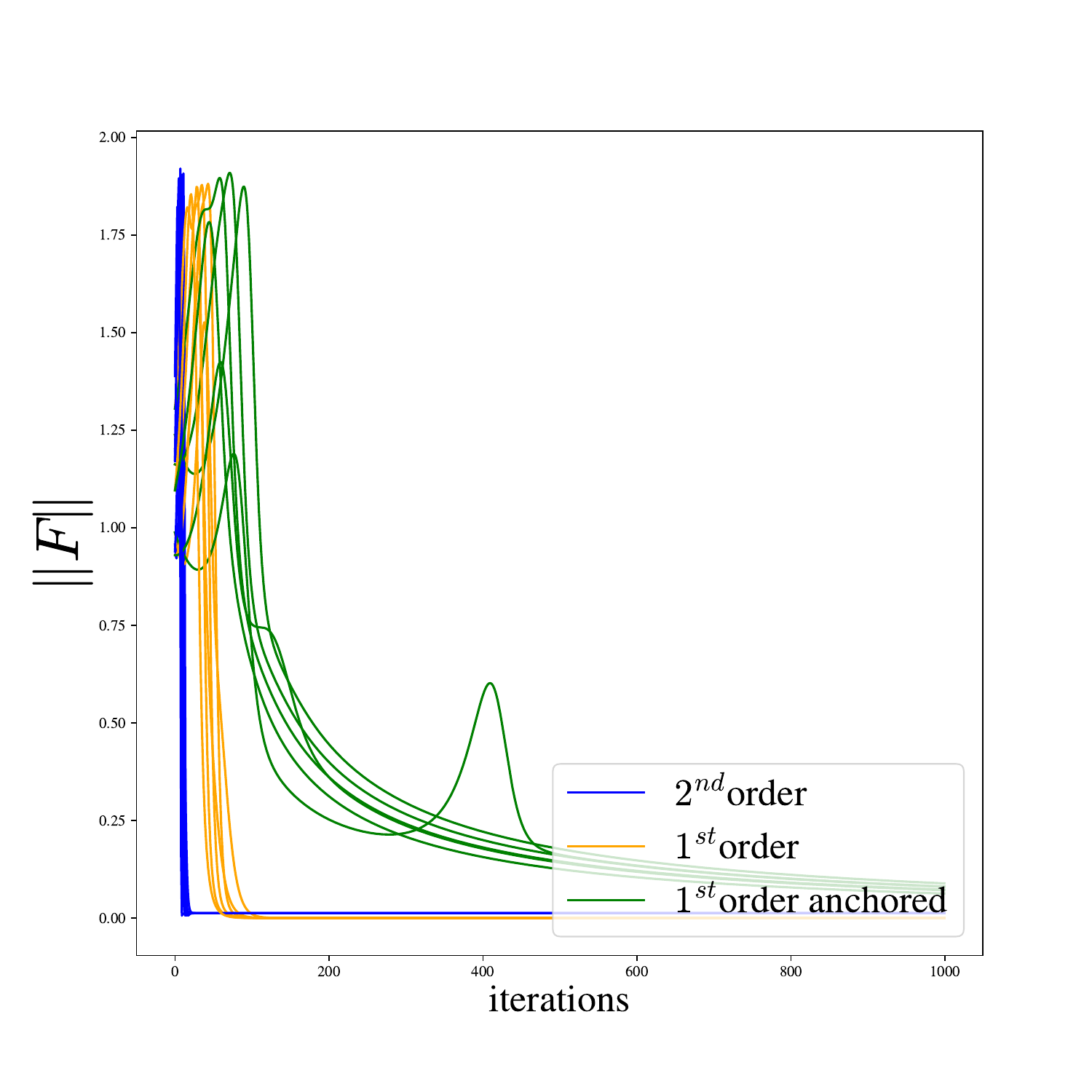}
        \\ (b) The operator norm $\|F\|$
    \end{minipage}%

    \caption{Visualization of algorithm performance on the modified forsaken example \ref{example:mforsaken}.}
    \label{figure:mforsaken}
\end{figure*}

To illustrate the performance of our method in the $\ell_2$ setting, we consider the following modified version \ref{example:mforsaken} of the example \ref{example:forsaken} (introduced by \citet{pethick2022escaping}), which provides a weak-MVI function in the Euclidean setting.

\begin{example}
\begin{equation}\label{example:mforsaken}\tag{Modified-Forsaken}
    \min_{|x|\leq 2} \max_{|y|\leq 2} f(x,y) = x(y-1.5)+h(x)-h(y)
\end{equation}
where $h(t) = \frac{t^2}{4}-\frac{t^4}{2}+\frac{t^6}{6}$.
\end{example}

 The \ref{example:forsaken} example was modified such that there exists a point in the domain that satisfies the weak-MVI condition for the operator $F=(\nabla_x f,-\nabla_y f)$, the original construction did not satisfy the weak-MVI for the operator with the range on $\rho$ specified by \cite{diakonikolas2021efficient}. We compare the performance of our method with the anchored first-order method in the monotone setting (EAG) \citep{yoon2021accelerated}. While EAG achieves the same rate of $O(\frac{1}{\epsilon})$ in the monotone setting (in which the norm is guaranteed to decrease monotonically), as compared to our second-order method (\hoeg, $p=2$), we observe that in the {weak}-MVI example the norm oscillates and the method converges slower than both the first and second order instances of \hoeg.

 We numerically verify that the real-valued stationary point of the Modified-Forsaken example is a weak-MVI solution to the problem. We illustrate the performance of the first- and second-order instances of our algorithm on this example in Figure \ref{figure:mforsaken}. This corresponds to using a $s=2$ and $s=1$ method on a problem which satisfies the first-order {weak}-MVI condition.

Note that while Theorem~\ref{theorem:balanced} shows convergence with operator norm decreasing at a rate of $O(1/{K^s})$ for the $s^{th}$-order instance of our algorithm on the $s^{th}$-order weak-MVI condition, we show convergence for some $s^{th}$-order instance on a $q^{th}$-order weak-MVI condition at a rate of $O(1/{K^s})$ for the operator norm in Theorem \ref{theorem:imbalanced}. Additional experiments are provided in Appendix \ref{app:experi}.

%% file: Sections/7.Conclusion.tex
\section{Conclusion}

We propose higher-order methods for min-max problems satisfying a certain weak-MVI condition, which allows us to show improvements in a larger range of $\rho$ for more general $\ell_p$ geometries. Our results highlight the possibilities of obtaining small operator norm beyond $\rho=0$, for geometries beyond the Euclidean setting, thereby addressing a challenging problem identified in previous work \citep{diakonikolas2021efficient}. We additionally provide algorithms to address both monotone and continuous-time settings

Furthermore, we give an algorithm that achieves a rate of $O({1}/{\epsilon^\frac{p}{s+1}})$ for the VI objective on the monotone setting in the $\ell_p$-geometry setting. We give a separate $\ell_2$ geometry specific algorithm that achieves a larger range of $\rho$ for small Lipschitz-continuity constant in the $\ell_2$ setting then the general $\ell_p$ algorithm. Finally, we analyze a dynamical system representing our $\ell_2$ geometry algorithm under the continuous-time limit.

%% file: Sections/Appendix.tex
\onecolumn
\section{Proofs for $\ell_p$-Geometry}\label{app:lphomvi}
\subsection{On the Bregman geometry of $\norm{\cdot}_p^s$}
We begin with the proof of the relation between $h(z_a-z_b) = \|z_a-z_b\|_p^s$ and $\omega_h(z_a,z_b)$ for any two points $z_a,z_b\in \cz$. In order to do so we first prove two supporting lemmas.

\begin{lemma}\label{lem:Euler}
    Consider the function $h(z) = \|z\|_p^s$. We have that $$\ang{\nabla^3 h(z)[u,u],z} = (s-2) \ang{\nabla^2 h (z)u,u}.$$
\end{lemma}
\begin{proof}
    Since \(h:\mathbb{R}^d \to \mathbb{R}\) is homogeneous of degree \(s\), i.e.,
\[
h(\lambda z) = \lambda^s h(z), \quad \forall \lambda > 0.
\]
Differentiating with respect to \(\lambda\) at \(\lambda=1\) gives
\[
\frac{d}{d\lambda} h(\lambda z) \Big|_{\lambda=1} = \langle \nabla h(z), z \rangle = s\, h(z).
\]
Differentiate the identity \(\langle \nabla h(z), z \rangle = s h(z)\) along direction \(u\):
\[
D\big(\langle \nabla h(z), z \rangle\big)[u] = D(s h(z))[u] = s \langle \nabla h(z), u \rangle.
\]
Compute the LHS using the product rule:
\[
D\big(\langle \nabla h(z), z \rangle\big)[u] = \langle \nabla^2 h(z) u, z \rangle + \langle \nabla h(z), u \rangle.
\]
Equating both sides:
\[
\langle \nabla^2 h(z) u, z \rangle + \langle \nabla h(z), u \rangle = s \langle \nabla h(z), u \rangle
\implies \langle \nabla^2 h(z) u, z \rangle = (s-1) \langle \nabla h(z), u \rangle.
\]
Differentiate \(\langle \nabla^2 h(z) u, z \rangle = (s-1) \langle \nabla h(z), u \rangle\) along direction \(u\):
\[
D\big(\langle \nabla^2 h(z) u, z \rangle\big)[u] = \langle \nabla^3 h(z)[u,u], z \rangle + \langle \nabla^2 h(z) u, u \rangle,
\]
We also have,
\[
D\big((s-1)\langle \nabla h(z), u \rangle\big)[u] = (s-1) \langle \nabla^2 h(z) u, u \rangle.
\]
Equating both sides and simplifying we obtain the statement of the lemma
\[
\langle \nabla^3 h(z)[u,u], z \rangle + \langle \nabla^2 h(z) u, u \rangle = (s-1) \langle \nabla^2 h(z) u, u \rangle
\implies \langle \nabla^3 h(z)[u,u], z \rangle = (s-2) \langle \nabla^2 h(z) u, u \rangle.
\]
\end{proof}
\begin{lemma}\label{lem:2sides}
With $h(z) = \|z\|_p^s$ we have that the Bregman divergence with potential function $h$ satisfies, $$\omega_{h}(z_a,z_b)+\omega_{h}(z_b,z_a) \geq \frac{4s}{2^s} \|z_a-z_b\|_p^s$$ for $s\geq p \geq 2$.
\end{lemma}
\begin{proof}
    We have,
    $$\omega_{h}(z_a,z_b)+\omega_{h}(z_b,z_a) = \ang{\nabla h (z_a)-\nabla h (z_b),z_a-z_b}.$$
    Consider the ratio,
    $$f(z_a,z_b) = \frac{\ang{\nabla h (z_a)-\nabla h (z_b),z_a-z_b}}{\|z_a-z_b\|_p^s}$$
    We have $\nabla h(z) = s\|z\|_p^{s-p} \cdot (|z_i|^{p-1}\text{sign}(z_i))_{i=1}^n$ (where $z_i$ is the $i^{th}$ component of the vector $z$ and (\(a_i)_{i=1}^n\) denotes the vector whose \(i^{\text{th}}\) component is \(a_i\)). Observe that $f(z_a,-z_a) = \frac{4s}{2^s}$ for $z_a \neq 0$.

    Let $q=\frac{z_a+z_b}{2},r = \frac{z_a-z_b}{2}$ and let $g_r(q) = f(q+r,q-r)$ . We will show that minimum of $g_r(q)$ occurs at $q=0$ for a fixed non-zero $r$. We will show it by contradiction.

    Let the minimum occur at $q^*\neq 0$. Define $\Phi_r(t,l) = g_r(\frac{tq^*}{l})$ for $t,l>0$. We will evaluate $\Phi_r'(t,l)$. Define $w_r(t,l) := \big\langle \nabla h(a(t,l)) - \nabla h(b(t,l)),\, r \big\rangle $, where $a(t,l) = \frac{t q^*}{l}+r$ and $b(t,l) = \frac{t q^*}{l}-r$, then we have:
\[
\Phi(t,s) = \frac{w_r(t,s)}{\|r\|_p^s}.
\]

Differentiate \(w(t,l)\) with respect to $t$. Since $\nabla a(t,l) = \nabla b(t,l) = \tfrac{q^*}{2s}$ we have,
\[
\begin{aligned}
\nabla w(t,s) 
&= \Big\langle \nabla^2 h(a(t,l))\,\nabla a(t,l) - \nabla^2 h(b(t,l))\,\nabla  b(t,l),\, r \Big\rangle \\[4pt]
&= \Big\langle \nabla^2 h(a(t,l))\frac{q^*}{l} - \nabla^2 h(b(t,l))\frac{q^*}{l},\, r \Big\rangle \\[4pt]
&= \frac{1}{l}\Big\langle \big(\nabla^2 h(a(t,l)) - \nabla^2 h(b(t,l))\big)q^*,\, r \Big\rangle.
\end{aligned}
\]

Also,
\[
\nabla^2 h(z) = \nabla^2 \|z\|_p^s
= s(s-p)\,\|z\|_p^{\,s-2p}\,z^{[p-1]}\big(z^{[p-1]}\big)^\top
\;+\; s(p-1)\,\|z\|_p^{\,s-p}\,\operatorname{diag}\!\big(|z_i|^{p-2}\big)_{i=1}^d
\]
 where, $
z^{[p-1]} := \big( |z_1|^{p-2} z_1, \; |z_2|^{p-2} z_2, \; \dots, \; |z_d|^{p-2} z_d \big)^\top.$ Plugging in the above expression we obtain,
$$
\begin{aligned}
\Phi'(t,l) = &\frac{1}{s}\Bigg\{
\Big\| {t\frac{q^*}{l} + r} \Big\|_p^{\,s-p}
\Big[ (p-1)\sum_i \Big|{t\frac{q^*_i}{l} + r_i}\Big|^{p-2} \frac{q^*_i}{l}\, r_i \\
&\qquad\qquad\qquad
+ (s-p)\,\frac{\Big(\sum_i \big|{t\frac{q^*_i}{l} + r_i}\big|^{p-2}({t\frac{q^*_i}{l} + r_i})\frac{q^*_i}{l}\Big)
\Big(\sum_i \big|{t\frac{q^*_i}{l} + r_i}\big|^{p-2}{(t\frac{q^*_i}{l} + r_i}) r_i\Big)}
{\Big\|{t\frac{q^*}{l} + r}\Big\|_p^{\,p}}
\Big] \\
&\quad
- \Big\| {t\frac{q^*}{l} - r} \Big\|_p^{\,s-p}
\Big[ (p-1)\sum_i \Big|{t\frac{q^*_i}{l} - r_i}\Big|^{p-2} \frac{q^*_i}{l}\, r_i \\
&\qquad\qquad\qquad
+ (s-p)\,\frac{\Big(\sum_i \big|{t\frac{q^*_i}{l} - r_i}\big|^{p-2}(t\frac{q^*_i}{l} - r_i)\frac{q^*_i}{l}\Big)
\Big(\sum_i \big|{t\frac{q^*_i}{l} - r_i}\big|^{p-2}{(t\frac{q^*_i}{l} - r_i}) r_i\Big)}
{\Big\|{t\frac{q^*}{l} - r}\Big\|_p^{\,p}}
\Big]
\Bigg\}
\end{aligned}
$$

Note that $\Phi'(0,l) = 0,\forall l>0$. Now consider $\Phi''(t,l)$. 

\begin{align*}
w''(t,s) 
&= \frac{1}{2} \Big( 
\Big\langle \nabla^3 h(a(t,l))\Big[a'(t,s), \frac{q^*}{l}\Big], r \Big\rangle 
- \Big\langle \nabla^3 h(b(t,l))\Big[b'(t,s), \frac{q^*}{l}\Big], r \Big\rangle 
\Big) \\[2mm]
&= \frac{1}{2} \Big( 
\Big\langle \nabla^3 h(a(t,l))\Big[2 \frac{q^*}{l}, \frac{q^*}{l}\Big], r \Big\rangle 
- \Big\langle \nabla^3 h(b(t,l))\Big[2 \frac{q^*}{l}, \frac{q^*}{l}\Big], r \Big\rangle 
\Big) \\[1mm]
&= \frac{1}{4} \Big( 
\Big\langle \nabla^3 h(a(t,l))\Big[\frac{q^*}{l}, \frac{q^*}{l}\Big], r \Big\rangle 
- \Big\langle \nabla^3 h(b(t,l))\Big[\frac{q^*}{l}, \frac{q^*}{l}\Big], r \Big\rangle 
\Big).
\end{align*}

This gives,
\[
\begin{aligned}
\text{At } t=0, \quad & a(0,l) = \frac{r}{2l}, \quad b(0,l) = -\frac{r}{2l}, \\[1mm]
\text{Thus, } f''(0) 
&= \frac{1}{4} \Big( 
\Big\langle \nabla^3 h\Big(\frac{r}{2}\Big)\Big[\frac{q^*}{l}, \frac{q^*}{l}\Big], r \Big\rangle 
- \Big\langle \nabla^3 h\Big(-\frac{r}{2}\Big)\Big[\frac{q^*}{l}, \frac{q^*}{l}\Big], r \Big\rangle 
\Big), \\[1mm]
\text{and therefore, }
\end{aligned}
\]

$$ \Phi''(0,l) 
= \frac{f''(0,l)}{\|r\|_p^s} 
= \frac{1}{4 \|r\|_p^s} \Big( 
\Big\langle \nabla^3 h\Big(\frac{r}{2}\Big)\Big[\frac{q^*}{l}, \frac{q^*}{l}\Big], r \Big\rangle 
- \Big\langle \nabla^3 h\Big(-\frac{r}{2}\Big)\Big[\frac{q^*}{l}, \frac{q^*}{l}\Big], r \Big\rangle 
\Big).$$

But since $h(z)$ is an even function we have that $\nabla^3 h(z)$ is an odd function. Thus,

$$ \Phi''(0,l) 
= \frac{f''(0)}{\|r\|_p^s} 
= \frac{1}{2 \|r\|_p^s} \Big( 
\Big\langle \nabla^3 h\Big(\frac{r}{2}\Big)\Big[\frac{q^*}{l}, \frac{q^*}{l}\Big], r \Big\rangle).$$

From Lemma \ref{lem:Euler}, we have,

\[
\big\langle \nabla^3 h\Big(\tfrac{r}{2}\Big)\Big[\frac{q^*}{l}, \frac{q^*}{l}\Big], r \Big\rangle
= 2 \, \big\langle \nabla^3 h\Big(\tfrac{r}{2}\Big)\Big[\frac{q^*}{l}, \frac{q^*}{l}\Big], \frac{r}{2} \Big\rangle
= \frac{2(s-2)}{l^2} \, \big\langle \nabla^2 h\Big(\tfrac{r}{2}\Big) q^*, q^* \Big\rangle.
\]

Which gives,
\[
\Phi''(0,l) 
= \frac{1}{2 \|r\|_p^s} \cdot 2 (s-2) \, \Big\langle \nabla^2 h\Big(\tfrac{r}{2}\Big) \frac{q^*}{l}, \frac{q^*}{l} \Big\rangle
= \frac{s-2}{l^2 \, \|r\|_p^s} \, \big\langle \nabla^2 h\Big(\tfrac{r}{2}\Big) q^*, q^* \big\rangle > 0
\]


strict inequality holds for $s>2$ since $\ang{q^*,r} \neq 0$. This follows from the fact that $\ang{q^*,r} = 0$ iff $q^*  =0$ and $q^* \neq 0$, thus $\Phi''(0,l) >0~\forall~s$. Thus we have a neighborhood $\mathcal{N}_\epsilon = [0,\epsilon]$ where $\Phi(t,l) \geq \Phi(0,l)~\forall~t\in \mathcal{N}_\epsilon$. Thus by definition we have,

$$ g(\epsilon \frac{q^*}{l}) > g(0),\forall~l.$$ Finally, setting $l = \epsilon$ we have $g(q^*)>g(0)$ for $s>2$. 

For $s=p=2$ we have $$f(z_a,z_b) = \frac{\ang{\nabla h (z_a)-\nabla h (z_b),z_a-z_b}}{\|z_a-z_b\|_p^s} = 2\frac{\ang{z_a-z_b,z_a-z_b}}{\|z_a-z_b\|_2^2} =2$$.

which is a constant and equal to the lower bound evaluated at $2$, i.e., $\frac{4s}{2^s}|_{s=2}$.

\end{proof}
\omegapnorm*
\begin{proof}
    From Lemma \ref{lem:2sides} we have $\ang{\nabla h(z_a) - \nabla h(z_b),z_a-z_b} \geq \frac{4s}{2^s}\|z_a-z_b\|_p^s$. This gives, 
\begin{align*}
h(z_b) - h(z_a) - \langle\nabla h(z_a), z_b - z_a\rangle &= \int_0^1 \langle h(z_a + \tau(z_b - z_a)) - \nabla h(z_a), z_b - z_a\rangle d\tau \\
&= \int_0^1 \frac{1}{\tau}\langle h(z_a + \tau(z_b - z_a)) - \nabla h(z_a), \tau(z_b - z_a)\rangle d\tau \\
&{\geq} \int_0^1 \frac{4s}{2^s} \tau^{s-1} \|z_b - z_a\|_p^s d\tau =  \frac{4}{2^s}\|z_b - z_a\|_p^s.
\end{align*}
\end{proof}

We now restate and provide the                                        proof of the key Lemma for the monotone $L_p$ condition. 

\subsection{Preliminary lemmas for $\ell_p$-HOMVI}

\begin{lemma}[\cite{tseng2008accelerated}]\label{lem:Tseng}
 Let $\phi$ be a convex function, let $z_a\in \cz$, and let
\[
z_c = \arg\min_{z'\in \cz}\{\phi(z') + \omega(z',z_a)\}.
\]
Then, $\forall$ $z_b\in \cz$, we have,
$ \phi(z_b) + \omega(z_b,z_a) \geq \phi(z_c) + \omega(z_c,z_a) +
\omega(z_b,z_c).$
\end{lemma}

\begin{lemma}[\cite{adil2022optimal}]\label{lem:mathfact}
Given \(R,\, \xi_1,\ldots,\xi_T \ge 0\) such that \(\sum_{t=1}^T \xi_t^2 \le R\), we have
\[
\sum_{t=1}^T \xi_t^{-q} \;\ge\; \frac{T^{\frac{q}{2}+1}}{R^{q/2}} .
\]
\end{lemma}

\begin{restatable}{lemma}{Supportive}\label{lemma:supportivelp}
Let $p\geq 2$, for the iterates of $\ell_p$-HOMVI (Algorithm~\ref{alg:mainalg1}), $\{z_k\}_{k=1}^K$, we have $\forall$ $z \in \mathcal{Z}$,
\[
\sum_{k=0}^K  \lambda_k \frac{s!}{L_{s,p}} \langle F(z_{k+\frac{1}{2}}),z_{k+\frac{1}{2} } -z\rangle \leq \omega(z, z_0) - (1-\frac{p-1}{p}2^{\frac{(s+1)(p+1)-pl}{p-1}}\frac{1}{p}^\frac{1}{p-1})\sum_{k=0}^K 
\omega(z_{k+\frac{1}{2}}, z_k).\]

Equivalently we have,
\[
\sum_{k=0}^K  \lambda_k \frac{s!}{L_{s,p}} \langle F(z_{k+\frac{1}{2}}),z_{k+\frac{1}{2} } -z\rangle \leq \omega(z, z_0) - (1-\frac{p-1}{p}2^{\frac{(s+1)(p+1)-pl}{p-1}}\frac{1}{p}^\frac{1}{p-1}) \sum_{k=0}^K 
(2 \lambda_k )^\frac{-p}{s+1-p}\]
\end{restatable}
\begin{proof}
For any $k$ and any $z \in \calZ$, we first apply Lemma~\ref{lem:Tseng} with $\phi(z) = \lambda_k \frac{s!}{L_{s,p}} \langle F(z_{k+\frac{1}{2}}),z - z_k\rangle$ to the first order update, which gives us
\begin{equation}\label{eq:Ksengsec3}
\lambda_k \frac{s!}{L_{s,p}} \langle F(z_{k+\frac{1}{2}}),z_{k+1} -z\rangle \leq \omega(z,z_k) - \omega(z,z_{k+1}) - \omega(z_{k+1},z_k). 
\end{equation}

Additionally, from the higher-order update of $\ell_p$-HOMVI we have
\begin{equation}
\label{eqn:MPstep1sec3}
\left\langle \mathcal{T}_{s-1}(z_{k+\frac{1}{2}}; z_k) , z_{k+\frac{1}{2}} - z_{k+1} \right\rangle  \leq \frac{2^\nu  L_{s,p}}{s!} \omega(z_{k+\frac{1}{2}}, z_k)^{\frac{s+1-p}{p}} \left\langle \nabla h(z_k) - \nabla h(z_{k+\frac{1}{2}}),  z_{k+\frac{1}{2}} - z_{k+1}  \right\rangle.
\end{equation}

Applying the Bregmann three point property (Lemma~\ref{lem:3point}) and the definition of $\lambda_k$ to Equation~\eqref{eqn:MPstep1sec3}, we have
\begin{equation}
\label{eqn:Kseng2sec3}
\lambda_k \frac{s!}{L_{s,p}}  \left\langle \mathcal{T}_{s-1}(z_{k+\frac{1}{2}}; z_k) , z_{k+\frac{1}{2}} - z_{k+1} \right\rangle  \leq \omega(z_{k+1}, z_k) - \omega(z_{k+1}, z_{k+\frac{1}{2}}) - \omega(z_{k+\frac{1}{2}}, z_k).
\end{equation}

Summing Eqs.~\eqref{eq:Ksengsec3} and \eqref{eqn:Kseng2sec3}, we obtain 
\begin{align}
  \label{eqn:one_kter_1sec3}
&\lambda_k \frac{s!}{L_{s,p}} \left( \langle F(z_{k+\frac{1}{2}}),z_{k+\frac{1}{2} } -z\rangle + \left\langle   \mathcal{T}_{s-1}(z_{k+\frac{1}{2}}; z_k) - F(z_{k + \frac{1}{2}}) , z_{k+\frac{1}{2}} - z_{k+1} \right\rangle \right) \nonumber \\
&\leq \omega(z,z_k) - \omega(z,z_{k+1}) - \omega(z_{k+1}, z_{k+\frac{1}{2}}) - \omega(z_{k+\frac{1}{2}}, z_k).
\end{align}

Now, we obtain 
\begin{align}\label{peterpaul}
\lambda_k &\frac{s!}{L_{s,p}} \left\langle   \mathcal{T}_{s-1}(z_{k+\frac{1}{2}}; z_k) - F(z_{k + \frac{1}{2}}) , z_{k+\frac{1}{2}} - z_{k+1} \right\rangle \nonumber\\
&\substack{(i) \\ \geq} - \lambda_k \frac{s!}{L_{s,p}} \norm{\mathcal{T}_{s-1}(z_{k+\frac{1}{2}}; z_k) - F(z_{k + \frac{1}{2}}) }_* \norm{ z_{k+\frac{1}{2}} - z_{k+1}} \nonumber\\
&\substack{(ii) \\ \geq} - \lambda_k \norm{z_{k+ \frac{1}{2}} -  z_k}^{s} \norm{ z_{k+\frac{1}{2}} - z_{k+1}} \nonumber\\
&\substack{(iii) \\ \geq} - \lambda_k \norm{z_{k+ \frac{1}{2}} -  z_k}^{s} 2^\frac{p+1}{p}\omega(z_{k+1}, z_{k+\frac{1}{2}})^{\frac{1}{p}}\nonumber\\ 
&\substack{(iv) \\ \geq} -\lambda_k 2^{s\frac{p+1}{p}}\omega_p(z_{k+\frac{1}{2}},z_k)^\frac{s}{p}2^\frac{p+1}{p}\omega(z_{k+1}, z_{k+\frac{1}{2}})^{\frac{1}{p}} \nonumber\\ 
&\substack{(v) \\ \geq} -\frac{1}{2^\nu } \omega(z_{k+\frac{1}{2}}, z_k)^{-\frac{s+1-p}{p}} 2^{s\frac{p+1}{p}}\omega_p(z_{k+\frac{1}{2}},z_k)^\frac{s}{p}2^\frac{p+1}{p}\omega(z_{k+1}, z_{k+\frac{1}{2}})^{\frac{1}{p}} \nonumber\\ 
&\geq  -2^{\frac{(s+1)(p+1)-pl}{p}} \frac{1}{p}^\frac{1}{p}\omega(z_{k+\frac{1}{2}}, z_k)^{\frac{p-1}{p}} \pa{p^{\frac{1}{p}}\omega(z_{k+1}, z_{k+\frac{1}{2}})^{\frac{1}{p}}} \nonumber\\
&\substack{(vi) \\ =} -\pa{2^{\frac{(s+1)(p+1)-pl}{p-1}}\frac{1}{p}^\frac{1}{p-1}\frac{p-1}{p}\omega(z_{k+\frac{1}{2}}, z_k) +\omega(z_{k+1}, z_{k+\frac{1}{2}})}
\end{align}

Here, $(i)$ used Holder's inequality and $(ii)$ used assumption~\ref{assmpt:smooth}. For $(iii)$ and $(iv)$ since $p\geq 2$, we set $(x,\Delta) = (z_{k+\frac{1}{2}},z_{k+1}-z_{k+\frac{1}{2}})$ and $(x,\Delta) = (z_k,z_{k+\frac{1}{2}}-z_k)$ respectively in Lemma 2.1 (\cite{adil2024convex}). We substitute the value of $\lambda_k$ in $(v)$ and finally in $(vi)$ we use the inequality ${ab} \leq \frac{a^p}{p}+\frac{b^q}{q}$ for $a,b \geq 0$ when $\frac{1}{p}+\frac{1}{q}=1$ (Young's inequality). Note the importance of the factors used in the Young's inequality in step $(iv)$ for inequality \eqref{peterpaul}. This factor allows us to cancel out the $\omega(z_{k+1},z_{k+\frac{1}{2}})$ terms, leaving us with only the $\omega(z_{k+\frac{1}{2}},z_k)$ terms which are used to bound $\|F(z_{k+\frac{1}{2}})\|$ in the subsequent steps. Combining with Eq.~\eqref{eqn:one_kter_1sec3} and rearranging yields
\begin{align*}
\lambda_k \frac{s!}{L_{s,p}} \langle F(z_{k+\frac{1}{2}} ),z_{k+\frac{1}{2} } -z\rangle &\leq \omega(z,z_k) - \omega(z,z_{k+1}) - (1-\frac{p-1}{p}2^{\frac{(s+1)(p+1)-pl}{p-1}}\frac{1}{p}^\frac{1}{p-1})\omega(z_{k+\frac{1}{2}}, z_k)~\forall~u
\end{align*}

Summing over all iterations $k$ yields, 
\[
\sum_{k=0}^K  \lambda_k \frac{s!}{L_{s,p}} \langle F(z_{k+\frac{1}{2}}),z_{k+\frac{1}{2} } -z\rangle \leq \omega(z, z_0) - (1-\frac{p-1}{p}2^{\frac{(s+1)(p+1)-pl}{p-1}}\frac{1}{p}^\frac{1}{p-1})\sum_{k=0}^K 
\omega(z_{k+\frac{1}{2}}, z_k).\]

Equivalently we have,
\[
\sum_{k=0}^K  \lambda_k \frac{s!}{L_{s,p}} \langle F(z_{k+\frac{1}{2}}),z_{k+\frac{1}{2} } -z\rangle \leq \omega(z, z_0) - (1-\frac{p-1}{p}2^{\frac{(s+1)(p+1)-pl}{p-1}}\frac{1}{p}^\frac{1}{p-1}) \sum_{k=0}^K 
(2^\nu  \lambda_k )^\frac{-p}{s+1-p}\]

\end{proof}
\subsection{Proof of convergence for $\ell_p$-HOMVI}
 \MainComplexity*
\begin{proof}
Let $S_K = \sum_{i=0}^K \lambda_i$ and $\hat{z} = \frac{\sum_{i=0}^K \lambda_k z_{k+\frac{1}{2}}}{S_k}$. We first note that, $\forall$ $z \in \mathcal{Z}$,
\begin{align*}
\langle F(z), \hat{z} - z \rangle 
&= \frac{1}{S_K} \sum_{k=0}^K \lambda_k \langle F(z), z_{k+\frac{1}{2}} - z \rangle \\
&\leq \frac{1}{S_K} \sum_{k=0}^K \lambda_k \langle F(z_{k+\frac{1}{2}}), z_{k+\frac{1}{2}} - z \rangle \quad \text{(by monotonicity of $F$)} \\
&\leq \frac{1}{S_K} \sum_{k=0}^K \lambda_k \langle F(z_{k+\frac{1}{2}}), z_{k+\frac{1}{2}} - z \rangle.
\end{align*}

Observe from Lemma~\ref{lemma:supportivelp} that
\[
\sum_{k=0}^K  \lambda_k  \langle F(z_{k+\frac{1}{2}}), z_{k+\frac{1}{2}} - z \rangle \leq \frac{L_{s,p}}{s!} \omega(z, z_0),
\]

and also (by setting $z=z^*$) we have,
\[
\sum_{k=0}^K 
(2^\nu  \lambda_k)^\frac{-p}{s+1-p} \leq \frac{\omega(z^*, z_0)}{(1-\frac{p-1}{p}2^{\frac{(s+1)(p+1)-pl}{p-1}}\frac{1}{p}^\frac{1}{p-1})}
\]

since $\ang{F(z_{k+\frac{1}{2}}),z_{k+\frac{1}{2}}-z^*} \geq 0$.

In Lemma \ref{lem:mathfact}, set $\xi_k = (2^\nu \lambda_k)^\frac{-p}{2(s+1-p)}$. Then we have,
\[
\sum_{k=0}^K \xi_k^2 = \sum_{k=0}^K 
(2^\nu  \lambda_k)^\frac{-p}{s+1-p} \leq  \frac{\omega(z^*, z_0)}{(1-\frac{p-1}{p}2^{\frac{(s+1)(p+1)-pl}{p-1}}\frac{1}{p}^\frac{1}{p-1})} =  R.
\]

It follows by setting $q = \frac{2(s+1-p)}{p}$ in the conclusion of ,
\[
2^\nu  S_K =2^\nu  \sum_{k=0}^K \lambda_k = \sum_{k=0}^K \xi_k^{\frac{2(p-(s+1))}{p}} \geq \frac{(K+1)^{\frac{s+1}{p}}}{R^\frac{s+1-p}{p}}.
\]

Thus,
\[
2^\nu  S_K  \geq \frac{(K+1)^{\frac{s+1}{p}}}{R^\frac{s+1-p}{p}}.
\]

Next, $\forall$ $z \in \mathcal{Z}$,
\begin{align*}
\langle F(z), \hat{z} - z \rangle 
&\leq  \frac{ L_{s,p}}{s! S_K} \, \omega(z, z_0) \\
&\leq \frac{2^\nu  L_{s,p}}{s!} \cdot  \omega(z, z_0) \cdot \frac{
R^\frac{s+1-p}{p}}{(K+1)^{\frac{s+1}{p}}}.\\
&\leq \frac{2^\nu  L_{s,p}}{s!} \cdot  \omega(z, z_0) \cdot \frac{
(\frac{\omega(z^*, z_0)}{(1-\frac{p-1}{p}2^{\frac{(s+1)(p+1)-pl}{p-1}}\frac{1}{p}^\frac{1}{p-1})})^\frac{s+1-p}{p}}{(K+1)^{\frac{s+1}{p}}}.
\end{align*}

This implies we have $\langle F(z), \hat{z} - z \rangle \leq  \epsilon$ in $K=\left[ \frac{\frac{2^\nu  L_{s,p}}{s!} \cdot \omega(z, z_0) \cdot \left(\frac{\omega(z^*, z_0)}{( 1-\frac{p-1}{p}2^{\frac{(s+1)(p+1)-pl}{p-1}} \frac{1}{p}^{\frac{1}{p-1}} )} \right)^{\frac{s+1-p}{p}}}{\epsilon} \right]^{\frac{p}{s+1}}-1$ iterations. To find the minima as $l$ varies, we must optimize the complexity in terms of $l$ under the constraint $1>\frac{p-1}{p}2^{\frac{(s+1)(p+1)-pl}{p-1}} \frac{1}{p}^{\frac{1}{p-1}}$. This minimum occurs at $l=\frac{(s+1)(p+1)}{p}-\log_2 p+\frac{p-1}{p},\log_2 s$. Plugging in this value we obtain, 

\begin{align*}
    K&=
2^{p+1}
p^{-\frac{p}{s+1}}
s^{\frac{p-1}{s+1}}
\left(\frac{s}{s+1-p}\right)^{\frac{s+1-p}{s+1}} \left[ \frac{\frac{L_{s,p}}{s!} \cdot \omega(z, z_0) \cdot \omega(z^*, z_0)^{\frac{s+1-p}{p}}}{\epsilon} \right]^{\frac{p}{s+1}}-1\\
&\leq 2^{p+1}
p^{-\frac{p}{s+1}}
s^{\frac{p-1}{s+1}}
\left(\frac{s}{s+1-p}\right)^{\frac{s+1-p}{s+1}} \left[ \frac{L_{s,p} }{s!\epsilon} \right]^{\frac{p}{s+1}}D-1\\
&= 2^{p+1}
p^{-\frac{p}{s+1}}
s^{\frac{s}{s+1}}
\left(\frac{1}{s+1-p}\right)^{\frac{s+1-p}{s+1}} \left[ \frac{L_{s,p} }{s!\epsilon} \right]^{\frac{p}{s+1}}D-1
\end{align*}

where $D = \max_{z\in \cz} \omega(z,z_0)$.

\end{proof}

\subsection{Proof of convergence for $\ell_p$-\textsc{hoeg+}}
\WeakMVI*

\begin{proof} Consider the first update,
$$\mathcal{T}_{s-1}(z_{k+\frac{1}{2}};z_k)+\nabla_u(\frac{2^\nu  L_{s,p}}{s!}\|u\|_p^{s+1})|_{u=z_{k+\frac{1}{2}}-z_k} = 0$$

as a consequence, 
$$\|\mathcal{T}_{s-1}(z_{k+\frac{1}{2}};z_k)\|_{p^*}=2^\nu \|\nabla_u(\frac{L_{s,p}}{s!}\|u\|_p^{s+1})|_{u=z_{k+\frac{1}{2}}-z_k}\|_{p^*} = 2^\nu \frac{(s+1)L_{s,p}}{s!}\|z_{k+\frac{1}{2}}-z_k\|_p^{s}.$$

From smoothness we have,
$$\|F(z_{k+\frac{1}{2}})-\mathcal{T}_{s-1}(z_{k+\frac{1}{2}};z_k)\|_{p^*} \leq \frac{L_{s,p}}{s!}\|z_{k+\frac{1}{2}}-z_k\|_p^s.$$

But from the triangle-inequality we have, 

$$\left\| F\left(z_{k+\frac{1}{2}} \right) \right\|_{p^*} \leq \|F(z_{k+\frac{1}{2}})-\mathcal{T}_{s-1}(z_{k+\frac{1}{2}};z_k)\|_{p^*}+\|\mathcal{T}_{s-1}(z_{k+\frac{1}{2}};z_k)\|_{p^*}.$$

Overall this gives,
\begin{align}\label{eqn:imp}
\left\| F\left(z_{k+\frac{1}{2}} \right) \right\|_{p^*}
&\leq \frac{((s+1)2^\nu +1) L_{s,p}}{s!} \left\| z_{k+\frac{1}{2}} - z_k \right\|_p^s
\end{align}

Consider the function $M_k(u) = \ang{\mathcal{T}_{s-1}(z_{k+\frac{1}{2}};z_k),u-z_k}+\frac{2^\nu L_{s,p}}{s!}\|u-z_k\|_p^{s+1}$. Let $\omega_M$ denote the Bregman divergence induced by the function $M_k(\cdot)$. We then have from the definition of Bregman divergence,
\begin{align}\label{eqn:MMM}
    M_k(z_{k+1}) &= M_k(z_{k+\frac{1}{2}}) + \left\langle \nabla M_k(z_{k+\frac{1}{2}}), z_{k+\frac{1}{2}} - z_k \right\rangle \nonumber + \omega_{M}(z_{k+1}, z_{k+\frac{1}{2}}) \\
    &\geq M_k(z_{k+\frac{1}{2}}) + 2^\nu m_s\frac{L_{s,p}}{s!} \|z_{k+1} - z_{k+\frac{1}{2}}\|_p^{s+1},
\end{align}

since from the first update we have that $\nabla_u M_{k}(z_{k+\frac{1}{2}})|=0$ and that $\omega_M(z_{k+1}, z_{k+\frac{1}{2}})\geq m_s\frac{L_{s,p}}{s!} \|z_{k+1} - z_{k+\frac{1}{2}}\|_p^{s+1}$ which follows by setting $h= \|\cdot-z_k\|_p^{s+1}$ in $\omega_h$ in Theorem \ref{lem:omegapnorm} as follows,
 $$\omega_{\|z-z_k\|_p^{s+1}}(z_a,z_b) = \omega_{\|z\|_p^{s+1}}(z_a-z_k,z_b-z_k) \geq m_s \|z_a-z_b\|_p^{s+1}$$

for $s+1\geq p\geq 2$ and observing that $\omega_{M(z)}(z_a,z_b)=2^\nu \frac{L_{s,p}}{s!}\omega_{\|z-z_k\|_p^{s+1}}(z_a,z_b)$.
 
Plugging in the definition of $M_k(u)$ in Eq.~\eqref{eqn:MMM} we have,
\begin{align}\label{eqn:higher1}
\left\langle \mathcal{T}_{s-1}(z_{k+\frac{1}{2}}; z_k), z_{k+\frac{1}{2}}-z_{k+1} \right\rangle &\leq \frac{2^\nu L_{s,p}}{s!}\Big( \|z_{k+1} - z_k\|_p^{s+1} - \|z_{k+\frac{1}{2}} - z_k\|_p^{s+1}\\
&- m_s \|z_{k+1} - z_{k+\frac{1}{2}}\|_p^{s+1} \Big)\nonumber
\end{align}

From the second update we have,
\begin{equation}\label{eqn:firstup}
\lambda_k \frac{s!}{L_{s,p}} \, \langle F(z_{k+\frac{1}{2}}), z_{k+1} - z \rangle 
\leq \omega_p(z, z_k) - \omega_p(z, z_{k+1}) - \omega_p(z_{k+1}, z_k),
\end{equation}
where $\omega_p$ denotes $\omega_{\|z\|_p^p}$.

Multiplying \eqref{eqn:higher1} with $\frac{s!\lambda_k}{L_{s,p}}$ both sides and summing with \eqref{eqn:firstup} we have,
\begin{align*}
\lambda_k \frac{s!}{L_{s,p}}(\left\langle \mathcal{T}_{s-1}(z_{k+\frac{1}{2}}; z_k),\, z_{k+\frac{1}{2}} - z_{k+1} \right\rangle 
+  \left\langle F(z_{k+\frac{1}{2}}),\, z_{k+1} - z \right\rangle) \notag \\
\leq 2^\nu  \lambda_k \left( \|z_{k+1} - z_k\|_p^{s+1} - \|z_{k+\frac{1}{2}} - z_k\|_p^{s+1} - m_s \|z_{k+1} - z_{k+\frac{1}{2}}\|_p^{s+1} \right) \notag \\
\quad +  \omega_p(u, z_k) - \omega_p(u, z_{k+1}) - \omega_p(z_{k+1}, z_k).
\end{align*}

This gives from Lemma 2.1 \cite{adil2024convex}, ($m_p = \frac{1}{2^{p+1}}$),
\begin{align*}
\lambda_k \frac{s!}{L_{s,p}}\big(\left\langle \mathcal{T}_{s-1}(z_{k+\frac{1}{2}}; z_k)-F(z_{k+\frac{1}{2}}),\, z_{k+\frac{1}{2}} - z_{k+1} \right\rangle 
+  \left\langle F(z_{k+\frac{1}{2}}),\, z_{k+\frac{1}{2}} - z \right\rangle\big) \notag \\
\leq 2^\nu  \lambda_k \left( \|z_{k+1} - z_k\|_p^{s+1} - \|z_{k+\frac{1}{2}} - z_k\|_p^{s+1} - m_s \|z_{k+1} - z_{k+\frac{1}{2}}\|_p^{s+1} \right) \notag \\
\quad + \omega_p(z,z_k)- \omega_p(u, z_{k+1}) -m_p \|z_{k+1} - z_k\|_p^p
\end{align*}

But we have,
\begin{align*}
&\frac{\lambda_k s!}{L_{s,p}}  \left\langle \mathcal{T}^{s-1}(z_{k+\frac{1}{2}}; z_k) - F(z_{k+\frac{1}{2}}), z_{k+\frac{1}{2}} - z_{k+1} \right\rangle \\
&\stackrel{(i)}{\geq} -\frac{\lambda_k s!}{L_{s,p}} \left\|\mathcal{T}^{s-1}(z_{k+\frac{1}{2}}; z_k) - F(z_{k+\frac{1}{2}})\right\|_{p^*} \left\|z_{k+\frac{1}{2}} - z_{k+1}\right\|_p \\
&\stackrel{(ii)}{\geq} -\lambda_k \left\|z_{k+\frac{1}{2}} - z_k\right\|_p^s \left\|z_{k+\frac{1}{2}} - z_{k+1}\right\|_p 
\end{align*}

Upon rearranging we obtain,

\begin{align*}
\lambda_k \frac{s!}{L_{s,p}}
\left\langle F(z_{k+\frac{1}{2}}),\, z_{k+\frac{1}{2}} - z \right\rangle
&\leq
2^\nu  \lambda_k
\Big(
\|z_{k+1} - z_k\|_p^{s+1}
- \|z_{k+\frac{1}{2}} - z_k\|_p^{s+1}  \\
&\qquad
- m_s \|z_{k+1} - z_{k+\frac{1}{2}}\|_p^{s+1}
\Big) \\
&\quad
+ \omega_p(z,z_k)
- \omega_p(z,z_{k+1})
- m_p \|z_{k+1} - z_k\|_p^p \\
&\quad
+ \lambda_k
\|z_{k+\frac{1}{2}} - z_k\|_p^s
\|z_{k+\frac{1}{2}} - z_{k+1}\|_p .
\end{align*}

We re-write the last term to prepare to use the Peter-Paul inequality,

\begin{align*}
\lambda_k \frac{s!}{L_{s,p}}( \left\langle F(z_{k+\frac{1}{2}}),\, z_{k+\frac{1}{2}} - z \right\rangle) \notag &\leq 2^\nu  \lambda_k \left( \|z_{k+1} - z_k\|_p^{s+1} - \|z_{k+\frac{1}{2}} - z_k\|_p^{s+1} - m_s \|z_{k+1} - z_{k+\frac{1}{2}}\|_p^{s+1} \right) \notag \\
&\quad + \omega_p(z,z_k)-\omega_p(z,z_{k+1}) - m_p\|z_{k+1} - z_k\|_p^p\\
&\quad + \lambda_k \pa{\frac{\left\|z_{k+\frac{1}{2}} - z_k\right\|_p^s}{(2^\nu m_s(s+1))^\frac{1}{s+1}}}\pa{\left\|z_{k+\frac{1}{2}} - z_{k+1}\right\|_p (2^\nu m_s(s+1))^\frac{1}{s+1}}.
\end{align*}

Using Peter-Paul on the last term such that a cancellation occurs, we have,
\begin{align*}
\lambda_k \frac{s!}{L_{s,p}}( \left\langle F(z_{k+\frac{1}{2}}),\, z_{k+\frac{1}{2}} - z \right\rangle) \notag &\leq
 2^\nu \lambda_k \left( \|z_{k+1} - z_k\|_p^{s+1} - \|z_{k+\frac{1}{2}} - z_k\|_p^{s+1} - m_s \|z_{k+1} - z_{k+\frac{1}{2}}\|_p^{s+1} \right) \notag \\
&\quad + \omega_p(z,z_k)-\omega_p(z,z_{k+1}) - m_p\|z_{k+1} - z_k\|_p^p\\
&\quad +\lambda_k\frac{{\left\|z_{k+\frac{1}{2}} - z_k\right\|_p^s}^\frac{s+1}{s}}{(2^\nu m_s(s+1))^\frac{1}{s}}\frac{s}{s+1} +2^\nu \lambda_k m_s\frac{\left\|z_{k+\frac{1}{2}} - z_{k+1}\right\|_p^{s+1}}{(s+1)}(s+1)
\end{align*}

This gives,
\begin{align*}
\lambda_k \frac{s!}{L_{s,p}}( \left\langle F(z_{k+\frac{1}{2}}),\, z_{k+\frac{1}{2}} - z \right\rangle) \notag &\leq
2^\nu  \lambda_k \left( \|z_{k+1} - z_k\|_p^{s+1} - \|z_{k+\frac{1}{2}} - z_k\|_p^{s+1}  \right) \notag \\
&\quad + \omega_p(z,z_k)-\omega_p(z,z_{k+1}) - m_p\|z_{k+1} - z_k\|_p^p\\
&\quad+\lambda_k\frac{s{\left\|z_{k+\frac{1}{2}} - z_k\right\|_p}^{s+1}}{2^\frac{\nu}{s} m_s^\frac{1}{s}(s+1)^{1+\frac{1}{s}}}
\end{align*}

Rearranging we have,
\begin{align}\label{eqn:combined4}
\lambda_k \frac{s!}{L_{s,p}}( \left\langle F(z_{k+\frac{1}{2}}),\, z_{k+\frac{1}{2}} - z \right\rangle)  &\leq
\left(2^\nu \lambda_k  \|z_{k+1} - z_k\|_p^{s+1}-  m_p \|z_{k+1} - z_k\|_p^{p}  \right)\\
&\quad + \omega_p(z,z_k)-\omega_p(z,z_{k+1})\notag\\
&\quad+\lambda_k {{\left\|z_{k+\frac{1}{2}} - z_k\right\|_p}^{s+1}} (\frac{s}{2^\frac{\nu}{s} m_s^\frac{1}{s}(s+1)^{1+\frac{1}{s}}}-2^\nu )\notag
\end{align}

We take a slight detour to derive a relation between $\|z_{k+1}-z_k\|_p$ and $\|z_{k+\frac{1}{2}}-z_k\|_p$. From the second update we have that 
\begin{align}\label{eqn:interim}
    F(z_{k+\frac{1}{2}}) = -\frac{L_{s,p}}{\lambda_k s!}(\nabla h(z_{k+1})-\nabla h(z_k)),
\end{align}
    
where $h(z) = \|z\|_p^p$. We also have that $\ang{\nabla h(z_{k+1})-\nabla h(z_k),z_{k+1}-z_k} \geq p m_p \|z_{k+1}-z_k\|_p^p$.
Using Cauchy Schwarz we obtain,
$$
\|\nabla h(z_{k+1})-\nabla h(z_k)\|_{p^*}
\;\ge\; \frac{\langle \nabla h(z_{k+1})-\nabla h(z_k),\,z_{k+1}-z_k\rangle}{\|z_{k+1}-z_k\|_p}
\;\ge\; pm_p\|z_{k+1}-z_k\|_p^{p-1}.
$$

Combining the above with \eqref{eqn:interim} we obtain, 
$$
\|F(z_{k+\frac{1}{2}})\|_{p^*}
\;\ge\; pm_p\frac{L_{s,p}}{\lambda_k s!}\|z_{k+1}-z_k\|_p^{p-1}.
$$

Along with Eq.~\eqref{eqn:imp} we obtain,
\begin{equation}\label{eq:imp3}
    \frac{((s+1)2^\nu +1) L_{s,p}}{s!} \left\| z_{k+\frac{1}{2}} - z_k \right\|_p^s\geq \|F(z_{k+\frac{1}{2}})\|_{p^*}
\;\ge\; pm_p\frac{L_{s,p}}{\lambda_k s!}\|z_{k+1}-z_k\|_p^{p-1}.
\end{equation}

We now use this relation in the following way. To cancel the first two terms in \eqref{eqn:combined4} we need to set $\lambda_k\leq \frac{m_p}{2^\nu }\|z_{k+1}-z_k\|_p^{p-(s+1)}$. Using \eqref{eq:imp3} we have, 
$$\|z_{k+1}-z_k\|_p \leq \pa{\frac{\lambda_k((s+1)2^\nu +1)}{pm_p} \|z_{k+1}-z_k\|_p^s}^\frac{1}{p-1}$$

Thus in terms of $\|z_{k+\frac{1}{2}}-z_k\|_p$, it is sufficient to set $\lambda_k$ such that, $$\lambda_k \leq \frac{m_p}{2^\nu }\|z_{k+1}-z_k\|_p^{p-(s+1)} \leq \frac{m_p}{2^\nu } \Big(\frac{((s+1)2^\nu +1)\lambda_k}{p m_p} \, \|z_{k+\frac{1}{2}} - z_k\|_p^{s} \Big)^\frac{p-(s+1)}{p-1},$$ which is equivalent to $\lambda_k \leq \frac{m_p^{\frac{s+1}{s}}}{2^\frac{\nu(p-1)}{s} }
\left(\frac{(s+1)2^\nu +1}{p}\right)^{\frac{p-(s+1)}{s}}
\left\| z_{k+\frac12}-z_k \right\|_p^{\,p-(s+1)}.$ Setting $\lambda_k$ to the upper bound and rearranging, but keeping the expressions in terms of $\lambda_k$ (we will plug the value in later) we obtain,

\begin{align*}\label{eqn:combined6}
\lambda_k \pa{\frac{s!}{L_{s,p}} \left\langle F(z_{k+\frac{1}{2}}),\, z_{k+\frac{1}{2}} - z \right\rangle + {{\left\|z_{k+\frac{1}{2}} - z_k\right\|_p}^{s+1}} (2^\nu -\frac{s}{2^\frac{\nu}{s} m_s^\frac{1}{s}(s+1)^{1+\frac{1}{s}}})} &\leq \omega_p(z,z_k)\\
&-\omega_p(z,z_{k+1})
\end{align*}

Using Eq.~\eqref{eqn:imp} we have $\| z_{k+\frac12} - z_k \|_p^{s+1}
\ge
\left( \tfrac{s!}{((s+1)2^\nu +1)L_{s,p}} \right)^{\frac{s+1}{s}}
\| F(z_{k+\frac12}) \|_{p^*}^{\frac{s+1}{s}}
$ and setting $z=z^*$ we obtain,

\begin{align}
\lambda_k \Bigg[
\frac{s!}{L_{s,p}}
\left\langle 
F\!\left(z_{k+\frac{1}{2}}\right),\, 
z_{k+\frac{1}{2}} - z^* 
\right\rangle
&+
\left(2^\nu -\frac{s}{2^\frac{\nu}{s} m_s^{\frac{1}{s}}(s+1)^{1+\frac{1}{s}}}\right)
\left( \frac{s!}{(s+1+2^\nu )\,L_{s,p}} \right)^{\frac{s+1}{s}}
\bigl\| F\!\left(z_{k+\frac{1}{2}}\right) \bigr\|_{p^*}^{\frac{s+1}{s}}
\Bigg] \\
&\leq
\omega_p(z^*,z_k)-\omega_p(z^*,z_{k+1}) .
\end{align}

Summing up over $k$ and using the weak-MVI condition we have,
\begin{align}\label{eqn:combined6b-sub}
\sum_{k=0}^K\lambda_k \Bigg[
\pa{
\left(2^\nu -\frac{s}{2^\frac{\nu}{s} m_s^{\frac{1}{s}}(s+1)^{1+\frac{1}{s}}}\right)
\left( \frac{s!}{((s+1)2^\nu +1)\,L_{s,p}} \right)^{\frac{s+1}{s}}-\frac{s!\rho}{2L_{s,p}}}
\bigl\| F\!\left(z_{k+\frac{1}{2}}\right) \bigr\|_{p^*}^{\frac{s+1}{s}}
\Bigg]
\notag\leq
\omega_p(z^*,z_0)
\end{align}

Let $c_{s,p,\rho} = 
\pa{
\left(2^\nu -\frac{s}{2^\frac{\nu}{s} m_s^{\frac{1}{s}}(s+1)^{1+\frac{1}{s}}}\right)
\left( \frac{s!}{((s+1)2^\nu +1)\,L_{s,p}} \right)^{\frac{s+1}{s}}-\frac{s!\rho}{2L_{s,p}}}
$, for $c_{s,p,\rho}>0$ we need $$\rho
<
2\left(2^\nu -\frac{s}{2^{\frac{\nu}{s}} m_s^{\frac{1}{s}}(s+1)^{1+\frac{1}{s}}}\right)
\left( \frac{s!}{((s+1)2^\nu +1)^{\,s+1}\,L_{s,p}} \right)^{\frac{1}{s}},$$

the maxima of the upper bound of $\rho$ occurs at $l\simeq \log_2(s(s+1))$.

Plugging in the value of $\lambda_k$ after observing, 
\begin{align*}
   \lambda_k &= \frac{m_p^\frac{s+1}{s} }{2^\frac{\nu(p-1)}{s} }\left(\frac{(s+1)2^\nu +1}{p}\right)^{\frac{p-(s+1)}{s}}
\|z_{k+\frac12}-z_k\|_p^{\,p-(s+1)}\\
            &\geq \frac{m_p^\frac{s+1}{s} }{2^{\frac{\nu(p-1)}{s}}}   \left(\frac{(s+1)2^\nu +1}{p}\right)^{\frac{p-(s+1)}{s}}
            \left(
            \frac{s!}{((s+1)2^\nu +1)L_{s,p}}
            \right)^{\frac{p-(s+1)}{s}}
            \|F(z_{k+\frac{1}{2}})\|_{p^*}^{\frac{p-(s+1)}{s}}\\
            &= \frac{m_p^\frac{s+1}{s} }{2^{\frac{\nu(p-1)}{s}}}\left(
            \frac{s!}{pL_{s,p}}
            \right)^{\frac{p-(s+1)}{s}}
            \|F(z_{k+\frac{1}{2}})\|_{p^*}^{\frac{p-(s+1)}{s}}
\end{align*}

which gives, 

$$\sum_{k=0}^K \frac{m_p^\frac{s+1}{s}}{2^\frac{\nu(p-1)}{s} } 
\left(
\frac{s!}{pL_{s,p}}
\right)^{\frac{p-(s+1)}{s}}
c_{s,p.\rho}
\|F(z_{k+\frac{1}{2}})\|_{p^*}^{\frac{p}{s}}
\le
\omega_p(z^*,z_0).$$

and we have $m_p = \frac{1}{2^{p+1}}$ and $m_s = \frac{1}{2^{s-1}}$.
Thus we have,
$$
 \sum_{k=0}^K  \left\|F(z_{k+\frac{1}{2}})\right\|_p^\frac{p}{s} \leq\frac{\omega_p(z^*,z_0)}{C_{s,p}c_{s,p,\rho}}
$$

where $C_{s,p} = \frac{1}{2^{\frac{(p+1)(s+1)+\nu(p-1)}{s}}} 
\left(
\frac{s!}{pL_{s,p}}
\right)^{\frac{p-(s+1)}{s}}$. 
\end{proof}

\section{Proofs for Discrete Time $\ell_2$-Geometry}\label{subsec:supportive-dt}

\subsection{Supporting Lemmas}

\begin{lemma}[Three Point Property]\label{lem:3point}
Let $\omega(z,z_b)$ denote the Bregman divergence of a function $h$. The \textit{three point property} states, for any $z_a,z_b,z_c \in \cz$,
\[
\langle \nabla h(z_b) - \nabla h(z_c), z_a-z_c \rangle = \omega(z_a,z_c) + \omega(z_c,z_b) - \omega(z_a,z_b).
\]
\end{lemma}

\begin{lemma}[Adapted from \citet{adil2022optimal}]\label{lemma:supportive}
For the iterates of the Algorithm \ref{alg:mainalg1} $\{z_k\}_{k=1}^K$  we have $\forall z\in \cz$,
\begin{align*}
    \sum_{k=0}^K \lambda_k \frac{s!}{L_{s,2}}\langle F(z_{k+\frac{1}{2}}),z_{k+\frac{1}{2}}-z \rangle \leq \|z_0 - z\|^2
    -(1- \frac{1}{2^{2\nu+2}})\sum_{k=0}^K\|z_k-z_{k+\frac{1}{2}}\|^2,\nonumber
\end{align*}
\end{lemma}

\begin{proof}
For any $k$ and any $z \in \calZ$, we first apply Lemma~\ref{lem:Tseng} with $\phi(z) = \lambda_k \frac{s!}{L_{s,2}} \langle F(z_{k+\frac{1}{2}}),z - z_k\rangle$, which gives us
\begin{equation}\label{eq:Tsengsec3}
\lambda_k \frac{s!}{L_{s,2}} \langle F(z_{k+\frac{1}{2}}),z_{k+1} -z\rangle \leq \omega(z,z_k) - \omega(z,z_{k+1}) - \omega(z_{k+1},z_k). 
\end{equation}

Additionally, the guarantee of assumption~\ref{assmpt:smooth} with $z = z_{k+1}$ yields
\begin{equation}
\label{eqn:MPstep1sec3}
\left\langle \tau_{s-1}(z_{k+\frac{1}{2}}; z_k) , z_{k+\frac{1}{2}} - z_{k+1} \right\rangle  \leq \frac{2^\nu  L_{s,2}}{s!} \omega(z_{k+\frac{1}{2}}, z_k)^{\frac{s-1}{2}} 
\left\langle \nabla h(z_k) - \nabla h(z_{k+\frac{1}{2}}),  z_{k+\frac{1}{2}} - z_{k+1}  \right\rangle.
\end{equation}

Applying the Bregman three point property (Lemma~\ref{lem:3point}) and the definition of $\lambda_k$ to Eq.~\eqref{eqn:MPstep1sec3}, we have
\begin{equation}
\label{eqn:Tseng2sec3}
\lambda_k \frac{s!}{L_{s,2}}  \left\langle \mathcal{T}_{s-1}(z_{k+\frac{1}{2}}; z_k) , z_{k+\frac{1}{2}} - z_{k+1} \right\rangle  
\leq \omega(z_{k+1}, z_k) - \omega(z_{k+1}, z_{k+\frac{1}{2}}) - \omega(z_{k+\frac{1}{2}}, z_k).
\end{equation}

Summing Eqs.~\eqref{eq:Tsengsec3} and Eq.~\eqref{eqn:Tseng2sec3}, we obtain 
\begin{align}
  \label{eqn:one_kter_1sec3}
&\lambda_k \frac{s!}{L_{s,2}} \left( \langle F(z_{k+\frac{1}{2}}),z_{k+\frac{1}{2} } -z\rangle 
+ \left\langle \tau_{s-1}(z_{k+\frac{1}{2}}; z_k) - F(z_{k + \frac{1}{2} }) , z_{k+\frac{1}{2}} - z_{k+1} \right\rangle \right) \nonumber \\
&\leq \omega(z,z_k) - \omega(z,z_{k+1}) - \omega(z_{k+1}, z_{k+\frac{1}{2}}) - \omega(z_{k+\frac{1}{2}}, z_k).
\end{align}

Now, we obtain 
\begin{align}\label{peterpaul}
\lambda_k &\frac{s!}{L_{s,2}} \left\langle \tau_{s-1}(z_{k+\frac{1}{2}}; z_k) - F(z_{k + \frac{1}{2} }) , z_{k+\frac{1}{2}} - z_{k+1} \right\rangle \nonumber\\
&\substack{(i) \\ \geq} - \lambda_k \frac{s!}{L_{s,2}} \norm{\tau_{s-1}(z_{k+\frac{1}{2}}; z_k) - F(z_{k + \frac{1}{2} }) }_* \norm{ z_{k+\frac{1}{2}} - z_{k+1}} \nonumber\\
&\substack{(ii) \\ \geq} - \lambda_k \norm{z_{k+ \frac{1}{2}} -  z_k}^s \norm{ z_{k+\frac{1}{2}} - z_{k+1}} \nonumber\\ 
&\substack{(iii) \\ \geq} -\frac{1}{2^\nu } \omega(z_{k + \frac{1}{2}}, z_k)^{-\frac{s-1}{2}} \omega(z_{k + \frac{1}{2}}, z_k)^{\frac{s}{2}} \omega(z_{k+1}, z_{k+\frac{1}{2}})^{\frac{1}{2}} \nonumber\\ 
&= - \frac{1}{2^\nu } \omega(z_{k + \frac{1}{2}}, z_k)^{\frac{1}{2}} \omega(z_{k+1}, z_{k+\frac{1}{2}})^{\frac{1}{2}} \nonumber\\
&\substack{(iv) \\ \geq} - \frac{1}{2^{2\nu+2}}  \omega(z_{k + \frac{1}{2}}, z_k) - \omega(z_{k+1}, z_{k+\frac{1}{2}})
\end{align}

Here, $(i)$ used Hölder's inequality, $(ii)$ used assumption~\ref{assmpt:smooth}, $(iii)$ used the $1$-strong convexity of $\omega$, and $(iv)$ used the Peter-Paul inequality. 

Combining with Eq.~\eqref{eqn:one_kter_1sec3} and rearranging yields
\begin{equation*}
\lambda_k \frac{s!}{L_{s,2}} \langle F(z_{k+\frac{1}{2}}),z_{k+\frac{1}{2} } -z\rangle 
\leq \omega(z,z_k) - \omega(z,z_{k+1}) -(1- \frac{1}{2^{2\nu+2}}) \omega(z_{k+\frac{1}{2}}, z_k). 
\end{equation*}

Summing over all iterations $k$ yields 
\[
\sum_{k=0}^K  \lambda_k \frac{s!}{L_{s,2}} \langle F(z_{k+\frac{1}{2}}),z_{k+\frac{1}{2} } -z\rangle 
\leq \omega(z, z_0) - (1- \frac{1}{2^{2\nu+2}})\sum_{k=0}^K 
\omega(z_{k + \frac{1}{2}}, z_k).
\]

Finally setting the potential function $h(x) = \|x\|^2$ in $\omega$ gives us the statement of the lemma.
\end{proof}

\UpperBoundLemma*
\begin{proof}
From \ref{assmpt:smooth} and noting that for any two vectors $a\in \mathbb{R}^d,b\in\mathbb{R}^d$ we have,
\begin{equation}
    \|a\|-\|b\|\leq \|a-b\|
\end{equation}
we have
\begin{equation}\label{eqn:lipimp}
    \|F(z_{k+\frac{1}{2}})\|\leq \mathcal{T}_{s-1}(z_{k+\frac{1}{2}};z_k)+\frac{L_{s,p}}{s!}\|z_{k+\frac{1}{2}}-z_k\|^s.
\end{equation}

The update rule of Algorithm \ref{alg:mainalg3} gives us, $$\mathcal{T}_{s-1} +\frac{2^\nu L_{s,p}}{s!}\|z_{k+\frac{1}{2}}-z_k\|^{s-1}(z_{k+\frac{1}{2}}-z_k) = 0.$$
Thus, we have $\|\mathcal{T}_{s-1}(z_{k+\frac{1}{2}};z_k)\| = \frac{2^\nu L_{s,p}}{s!}\|z_{k+\frac{1}{2}}-z_k\|^s$. The statement of the lemma follows by combining with Eq.~\eqref{eqn:lipimp}.
\end{proof}

\subsection{Discrete Time : Main Results }

We now restate the convergence result of our discrete-time algorithm for completion.
\Balanced*
\subsection{$q^{th}$ \textit{weak}-MVI.}\label{app:qth}
We now propose and prove convergence of \hoeg~for the $s^{th}$-order instance of the Algorithm \ref{alg:mainalg3} under a $q^{th}$ order weak-MVI condition which is decoupled from the order $s$. This allows us to obtain guarantees for higher order methods when run on the original weak-MVI condition with $s=1$ as studied in the experiments of Section \ref{sec:experiments}. 
\begin{assumption}[$q^{th}$ Weak \textsc{mvi}]

There exists $z^* \in \cz^*$ such that:
\begin{equation}\label{assmpt:imbalanced}\tag{\textsc{a}$_3$}
    (\forall z \in \mathbb{R}^d):\quad 
    \innp{F(z), z - z^*} \geq -\frac{\rho}{2} \|F(z)\|^q,
\end{equation}
for some parameter $\rho$.
\end{assumption}
\begin{theorem}\label{theorem:imbalanced}\av{Fix constants.}
For any $p^{th}$-order smooth operator $F$ with bounded norm, $\|F(z)\| \leq D~\forall z \in \cz$, satisfying assumption~\ref{assmpt:imbalanced} with 
 $\rho\leq\frac{15}{16D^q}\frac{p!D}{L_p}^{\frac{p+1}{p}}$, when running a $p^{th}$-order instance of our Algorithm~\ref{alg:mainalg3} we have $\forall$ $k\geq 1:$ 
    $$
        \frac{1}{k+1}\sum_{k=0}^K \|F(z_{k+\frac{1}{2}})\|^\frac{2}{s} \leq \frac{c \|z_0 - z^*\|^2}{k+1}.
    $$
    In particular, we have that 
    \begin{align*}     &\min_{0\leq k\leq K} \|F(z_{k+\frac{1}{2}})\|^\frac{2}{s} \leq \frac{c\|z_0 - z^*\|^2}{K+1}\rightarrow \min_{0\leq k\leq K} \|F(z_{k+\frac{1}{2}})\|^2 \leq \frac{c_1}{(K+1)^p}.
    \end{align*}
\end{theorem}

\begin{proof}

Let $z^*$ be a SVI solution that satisfies the \textit{weak}-MVI condition. Setting $z=z^*$ in Lemma \ref{lemma:supportive}, for Algorithm \ref{alg:mainalg3},  we have,
\begin{align}\label{eq:mainf}
    \sum_{k=0}^K \lambda_k \frac{p!}{L_p}\langle F(z_{k+\frac{1}{2}}),z_{k+\frac{1}{2}}-z^* \rangle \leq\|z_0 - z^*\|^2-(1-\frac{1}{2^{2\nu+2}})\sum_{k=0}^K\|z_k-z_{k+\frac{1}{2}}\|^2
\end{align}

Adding $\sum_{k=1}^K\rho\lambda_k\|F(z_k)\|^q$ to both sides of Eq.~\eqref{eq:mainf}, and noting from the result of Lemma~\ref{lemma:upper-b} that $\|F(z_{k+\frac{1}{2}})\|\leq \frac{(1+2^\nu )L_{s,2}}{s!}\|z_{k+\frac{1}{2}}-z_k\|^s$ we get,
\begin{align*}
\rho\lambda_k\|F(z_{k+\frac{1}{2}})\|^q &= \rho \|F(z_{k+\frac{1}{2}})\|^{(q-\frac{s+1}{s})} \frac{\|F(z_{k+\frac{1}{2}})\|^{\frac{s+1}{s}}}{2\|z_k-z_{k+\frac{1}{2}}\|^{s-1}}\\
&\leq \frac{\rho}{2} (\frac{(2^\nu +1)L_{s,2}}{s!})^{\frac{s+1}{s}}\|z_k-z_{k+\frac{1}{2}}\|^2 \|F(z_{k+\frac{1}{2}})\|^{(q-\frac{s+1}{s})}
\end{align*}
 Further substituting $p=1$ in and choosing $x,y$ appropriately from \ref{assmpt:smooth} we obtain:
$$\|F(z_{k+\frac{1}{2}})\| \leq L_1\|z_{k+\frac{1}{2}}-z^*\|$$
further since $D = \max_k \|F(z_{k+\frac{1}{2}})\|$ we have,
\begin{align*}
\rho (\frac{L_p}{p!})^{\frac{s+1}{s}}\|z_k-z_{k+\frac{1}{2}}\|^2 \|F(z_{k+\frac{1}{2}})\|^{(q-\frac{s+1}{s})} \leq \rho D^q(\frac{(1+2^\nu )L_{s,2}}{s!D})^{\frac{s+1}{s}}\|z_k-z_{k+\frac{1}{2}}\|^2\nonumber
\end{align*}
Thus,
\begin{align}\label{main3}
    \sum_{k=0}^K \lambda_k (\frac{s!}{L_{s,2}}\langle F(z_{k+\frac{1}{2}}),z_{k+\frac{1}{2}}-z^* \rangle &+ \rho\|F(z_{k+\frac{1}{2}})\|^q)  \leq \|z^*-z_0\|^2\\
    &-(1-\frac{1}{2^{2\nu+2}}-\rho D^q(\frac{(1+2^\nu )L_{s,2}}{p!D})^{\frac{s+1}{s}})\sum_{k=0}^K\|z_k-z_{k+\frac{1}{2}}\|^2\notag
\end{align}
From assumption~\ref{assmpt:imbalanced} we have that LHS in Eq.~\eqref{main3} is non-negative. Setting $c_1 = (1-\frac{1}{2^{2\nu+2}}-\rho D^q(\frac{(1+2^\nu )L_{s,2}}{s!D})^{\frac{s+1}{s}})$, $ c_2 = (\frac{s!}{(1+2^\nu )L_{s,2}})^\frac{2}{s}$ and $\rho < D^{\frac{s+1}{s}-q}(1-\frac{1}{2^{2\nu+2}})\frac{s!}{(1+2^\nu )L_{s,2}})^\frac{s+1}{s}$ and using Lemma~\ref{lemma:upper-b} we get,
\begin{equation*}
   c_2\sum_{k=0}^K\|F(z_{k+\frac{1}{2}})\|^{\frac{2}{s}} \leq \sum_{k=0}^K\|z_k-z_{k+\frac{1}{2}}\|^2\leq \frac{\|z^*-z_0\|^2}{c_1}
\end{equation*}
And thus we obtain results analogous to Theorem~\ref{theorem:balanced}. Furthermore, setting $q=2$ gives us rates for the $s^{th}$-order versions of \hoeg~for the weak-MVI condition.
\end{proof}
Note that in the monotone setting, \citet{yoon2021accelerated} obtain faster rates for the $p=1$ case by using the anchoring technique used by \citet{diakonikolas2020halpern} to obtain the same rate for a subclass of monotone problems. In particular, they obtain $\|F(z_K)\|^2 \leq O\pa{1/{K^2}}$ while our algorithm guarantees a rate of $\min_{0\leq k\leq K}\|F(z_{k+\frac{1}{2}})\|^2 \leq O(1/{K^p})$, which is not tight for $p=1$. Whether we may obtain faster rates for higher-order methods in the monotone setting remains an open problem.

\input{Sections/6.Extensions}
\subsection{Additional experiments}\label{app:experi}

\begin{example}
\begin{equation}\label{example:forsaken}\tag{Forsaken}
    \min_{|x|\leq \frac{3}{2}} \max_{|y|\leq \frac{3}{2}} f(x,y) = x(y-0.45)+h(x)-h(y)
\end{equation}
where $h(t) = \frac{t^2}{4}-\frac{t^4}{2}+\frac{t^6}{6}$.
\end{example}

We empirically study the Forsaken example in \citet{pethick2022escaping} and show that while solving for the operator $F=(\nabla_x f,-\nabla_y f)$ results in the cycling of both the 1st and 2nd order methods as can be seen in Figure \ref{figure:forsaken}, using $F_\alpha$ allows us to converge to the only stationary point of the function, $z^* = (0.0780,0.4119)$. This happens because the stationary point $z^*$ does not satisfy the $weak$-MVI condition for $F$, though it satisfies (as we numerically verify) the \textit{weak}-MVI (and MVI) condition for $F_\alpha$, $\alpha\geq 2$. We further observe that the second order method when solving the variational inequality with $F_\alpha$ converges faster than the first order method.

Note that the first-order method is, in a certain sense, between first- and second-order, while the second-order method is similarly between second- and third-order in terms of the order of information of the derivative oracles.  The second-order method can thus be thought of as a higher order \textsl{oCGO} (\cite{vyas2023competitive}). We provide another example (Example \ref{example:x2y}) where our algorithm runs into issues when using $F$ but works well when using $F_\alpha$ in the Appendix.

\paragraph{Competitive operator $F_\alpha$.}
We also discuss natural extensions to our method that can be used to solve a larger class of saddle point problems. We show that our method can be extended to include the parameterized competitive operator $F_\alpha$ introduced in \citet{vyas2023competitive}. This operator is important as following it allows us to obtain small operator norm for a different generalization of the MVI condition, $\alpha$-MVI \citep{vyas2023competitive}. For this operator,
\begin{equation*}F_\alpha =  \begin{bmatrix}
I & \alpha \nabla_{xy}f\\
-\alpha \nabla_{yx}f & I
\end{bmatrix}^{-1}\begin{bmatrix}
\nabla_x f\\
-\nabla_y f 
\end{bmatrix}
\end{equation*} the \textit{weak}-MVI assumption~\eqref{assmpt:balanced} for $p=1$ contains the $\alpha$-MVI class. Note that if a operator $F$ satisfies the $\alpha$-MVI condition, the corresponding operator $F_\alpha$ satisfies the MVI condition. Furthermore the exact stationary points of $F_\alpha$ are the same as those of $F$ and an additional requirement of at least one SVI solution ($|\cal{Z^*}| \neq \phi$) with $F_\alpha$ as the operator is satisfied if a solution to SVI with $F$ as the operator exists. As a corollary of Theorem \ref{theorem:balanced} we have the following corollary describing the convergence of the operator $\|F_{\alpha}\|$ when $F_{\alpha}$ is used as the operator in our algorithm \hoeg. 

\begin{corollary}
Let the competitive field $F_{\alpha}$ satisfy the smoothness and weak-MVI assumptions \ref{assmpt:balanced}, \ref{assmpt:smooth} then for algorithm \ref{alg:mainalg3} solving for $F_{\alpha}$ we have $\forall$ $k\geq 1$,
    $$
        \frac{1}{K+1}\sum_{k=0}^K \|F_{\alpha}(z_{k+\frac{1}{2}})\|^\frac{2}{s} \leq \frac{ \|z_0 - z^*\|^2}{c_2(K+1)}.
    $$
    In particular, we have that 
    \begin{align*}
        \min_{0\leq k\leq K} \|F_{\alpha}(z_{k+\frac{1}{2}})\|^\frac{2}{s} \leq \frac{\|z_0 - z^*\|^2}{c_2(K+1)}\rightarrow \min_{0\leq k\leq K} \|F_{\alpha}(z_{k+\frac{1}{2}})\|^2 \leq \frac{C}{(K+1)^p} 
    \end{align*}

    where $C=\frac{\|z_0-z^*\|^{2p}}{c_2^p}$.
\end{corollary}

\begin{figure}[!ht]
    \centering
    \begin{minipage}{0.45\textwidth}
        \centering
        \includegraphics[width=\linewidth]{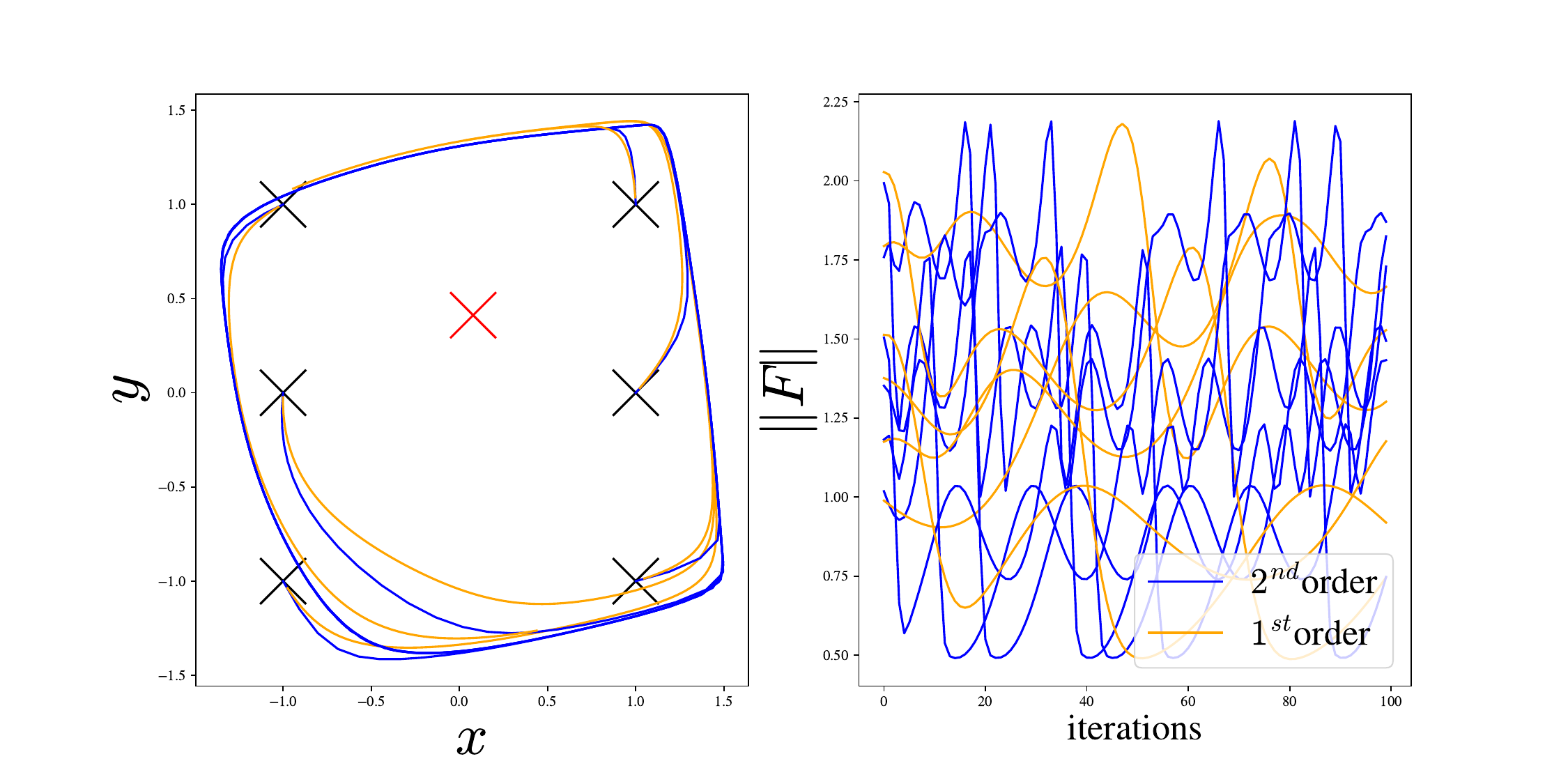}
        \\ (a) $F = (\nabla_x f,-\nabla_y f)$
    \end{minipage}%
    \hfill
    \begin{minipage}{0.45\textwidth}
        \centering
        \includegraphics[width=\linewidth]{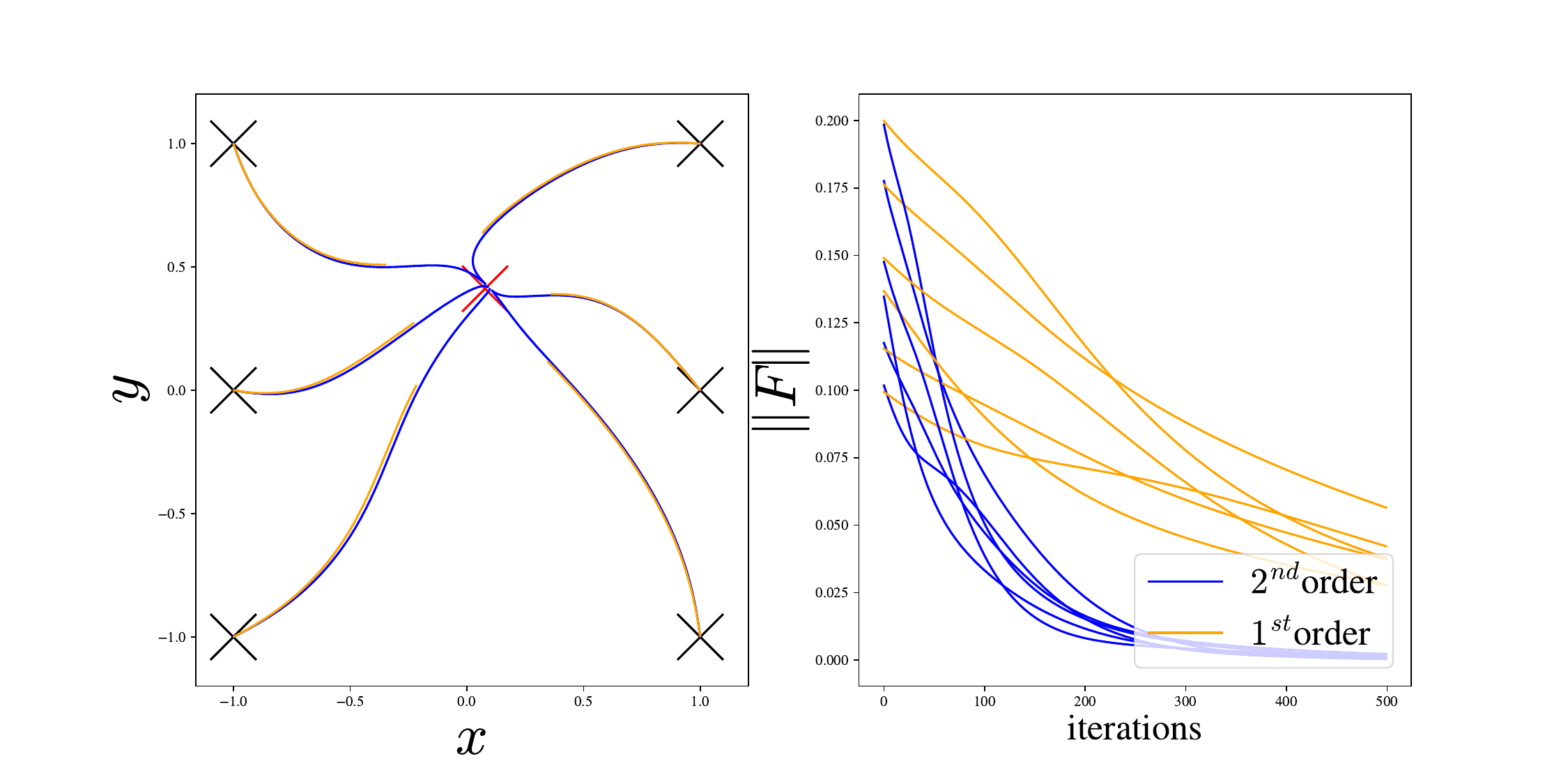}
        \\ (b) $F_\alpha:~\alpha = 10$
    \end{minipage}

    \caption{First and second-order methods with $F$ and $F_\alpha$ on (Forsaken).  
    While the algorithm using $F$ cycles, the algorithm using $F_\alpha$ converges to a stationary point.  
    A step size based on a Lipschitz constant of $L_1=10$ and $L_2=500$ is used for the first and second-order methods, respectively.}
    \label{figure:forsaken}
\end{figure}
\begin{example}\label{example:x2y}
We consider the problem  $$\min_{x\in\cx}\max_{y \in \cy} x^2y$$
 
Note that while all points on the $y$-axis are stationary points of the saddle point problem generated by $f(x,y)=x^2y$, there is only one global Nash equilibrium at the origin. However, when \hoeg~uses $F$ as the operator the iterates converge to different points on $x=0$ as can be seen in Figures~\ref{figure:x2y}(a), \ref{figure:x2y}(b). Using $F_\alpha$ as the operator mitigates this issue of convergence to non-Nash stationary points and as $\alpha$ increases, iterates from all the different initialization converge to the origin which is the Nash equilibrium, Figure~\ref{figure:x2y}(b). Experiments were run on a Macbook laptop computer. Theoretically mapping out the relation between the nature of stationary points and the operator used in this method remains an open direction of research.
\end{example}

\begin{figure}[h]
    \centering
    \begin{minipage}{0.45\textwidth}
        \centering
        \includegraphics[width=\linewidth]{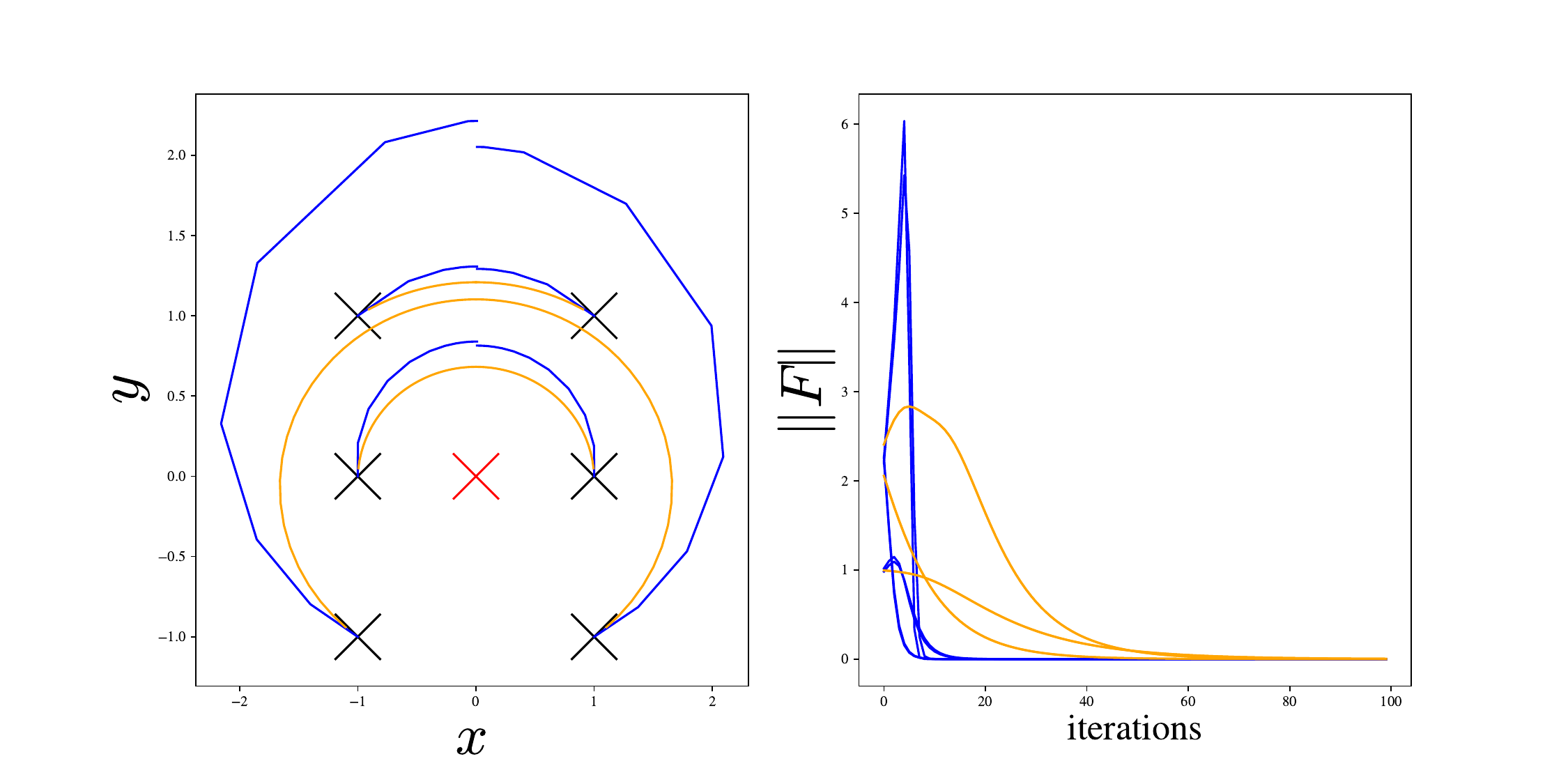}
        \\ (a) $F = (\nabla_x f,-\nabla_y f)$
    \end{minipage}%
    \hfill
    \begin{minipage}{0.45\textwidth}
        \centering
        \includegraphics[width=\linewidth]{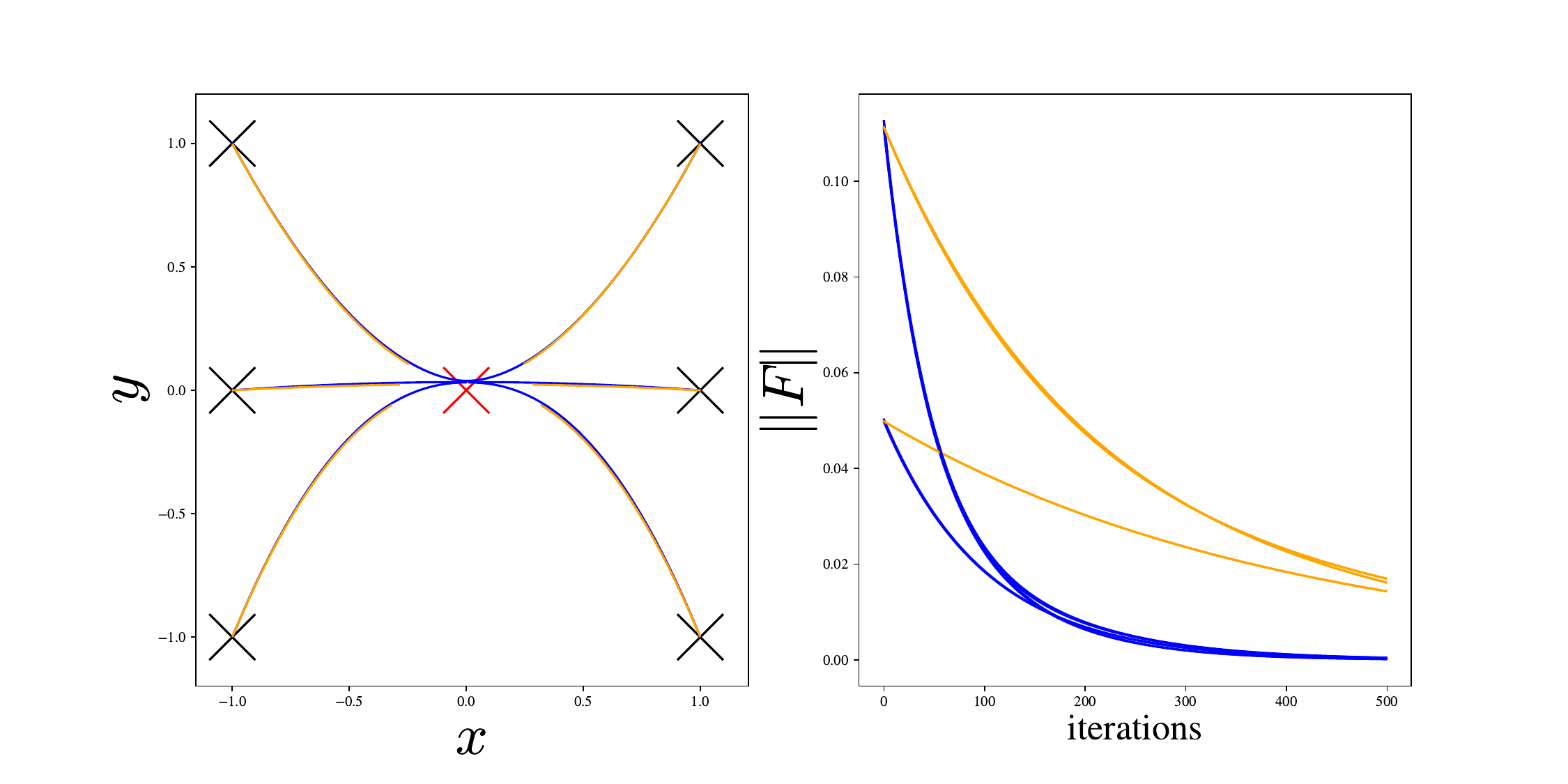}
        \\ (b) $F_{\alpha}:~\alpha = 10$
    \end{minipage}

    \caption{First and second-order methods with $F_{\alpha}$ and $F$.  
    (a) $F = (\nabla_x f,-\nabla_y f)$.  
    (b) $F_{\alpha}:~\alpha = 10$.}
    \label{figure:x2y}
\end{figure}
\section{Proofs for continuous-time $\ell_2$-Geometry.}

We begin with the theorem and proof of our continuous-time result.

\Continuoustime*
\begin{proof}
Let us consider the Lyapunov function $\mathcal{E} = \|s(t)\|^2$.

Expanding this equality with any $h \in \cz $, we have
\begin{align}\label{eq:lyapunov-derivative}
\frac{d\mathcal{E}(t)}{dt} &= 2\ang{u(t),\dot{u}(t)} =2\ang{\tfrac{F(z(t))}{\|F(z(t))\|^{1-1/s}},z_0-v(t)}\\
&=2\pa{\frac{\langle F(z(t)), z_0 - h\rangle}{\|F(z(t))\|^{1 - 1/s}}+\frac{\langle F(z(t)), h - z(t)\rangle\rangle}{\|F(z(t))\|^{1 - 1/s}}+\frac{\langle F(z(t)), z(t) - v(t)\rangle}{\|F(z(t))\|^{1 - 1/s}}}\nonumber
\end{align}

Since $\dot{u}(t) = - \tfrac{F(z(t))}{\|F(z(t))\|^{1-1/s}}$, we have
\begin{equation*}
\frac{\langle F(z(t)), z_0 - h\rangle}{\|F(z(t))\|^{1 - 1/s}} = -\langle \dot{u}(t), z_0 - h\rangle. 
\end{equation*}

Since $F$ satisfies the weak-MVI assumption, we have
$$\langle F(z(t)),h-z(t)\rangle \leq \frac{\rho}{2}\|F(z(t))\|^\frac{s+1}{s}-\langle F(z(t)),z^*-h \rangle.$$

Since $z(t) - v(t) + \tfrac{F(z(t))}{\|F(z(t))\|^{1-1/s}} = \textbf{0}$, we have
\begin{equation*}
\frac{\langle F(z(t)), z(t) - v(t)\rangle}{\|F(z(t))\|^{1 - 1/s}} = -\|F(z(t))\|^{\frac{2}{s}}. 
\end{equation*}

Plugging these pieces together in Eq.~\eqref{eq:lyapunov-derivative} yields that, for any $h \in \cz$, we have
\begin{align}\label{inequality:nonasymptotic-second}
\frac{d\mathcal{E}(t)}{dt} &\leq 2\langle \dot{u}(t), h-z_0\rangle + 2\langle  \dot u, z^* - h\rangle- (2-\rho)\|F(z(t))\|^{\frac{2}{s}}\\
&\leq 2\langle \dot{u}(t), z^*-z_0 \rangle - (2-\rho)\|F(z(t))\|^{\frac{2}{s}}.\nonumber
\end{align}

Rearranging we obtain
\begin{align*}
\|F(z(r))\|^{\frac{2}{s}} \; ds &\leq \frac{2}{2-\rho}\pa{\langle  \dot u(t), z^*-z_0\rangle - \frac{1}{2}\frac{d\mathcal{E}(t)}{dt}}.
\end{align*}

Upon integrating and observing $\mathcal{E}(0) = 0$ we obtain
\begin{align*}
\int_0^t \|F(z(r))\|^{\frac{2}{s}} \; ds &\leq \frac{2}{2-\rho}\pa{\langle s(t), z^*-z_0\rangle - \|u(t)\|^2}.
\end{align*}

Finally observing $\langle u(t), z^*-z_0\rangle - \|u(t)\|^2 \leq \frac{1}{4}\|z^*-z_0\|^2$, from Cauch Schwarz inequality and bounding $\|z_0 - z^*\|$ we conclude
\begin{align*}
\int_0^t \|F(z(r))\|^{\frac{2}{s}} \; dr \leq \frac{D^2}{2(2-\rho)}, \quad \textnormal{$\forall$ } t \geq 0.  
\end{align*}

Now let $$m= \min_{0 \leq r\leq t}  \|F(z(r))\|^{\frac{2}{s}} = (\min_{0 \leq r\leq t}  \|F(z(r))\|)^{\frac{2}{s}},$$ where the second equality is true since $\|x\|^\frac{2}{s}$ is increasing in $x$ for $s\geq 1$. 
Then we have, $\forall$ $t \geq 0$,
\begin{align*}
mt \leq \int_0^t \|F(z(r))\|^{\frac{2}{s}} \; dr \leq \frac{D^2}{2(2-\rho)}\rightarrow \min_{0 \leq r\leq t}  \|F(z(r))\|^2 \leq \frac{D^{2s}}{2^s(2-\rho)^st^s}\leq O\pa{\frac{1}{t^s}},
\end{align*}

which is the statement of the theorem.

\end{proof}

We now restate and provide the proof of Corollary \ref{cor:comonotone}.

\Comonotone*
\begin{proof}
Since the operator is $\rho$-comonotone with $\rho >-1$ we have
$$\langle F(z_1)-F(z_2),z_1-z_2 \rangle > -\|F(z_1)-F(z_2)\| ~\forall z_1,z_2 \in \cz$$
Setting $z_2=z^*$ where $\|F(z^*)\|=0$, we obtain that $F$ satisfies assumption \ref{assmpt:balanced} with $\rho < 2$ and thus from Theorem \ref{thm:continuoustime} we have,
\begin{align}\label{eqn:min}
    \min_{0 \leq r\leq t}  \|F(z(s))\|^2 \leq O(\frac{1}{t})
\end{align}
Furthermore setting $z_1 = z +\tau \delta z,~z_2 = z$ ($\tau>0$) and dividing both sides by $\tau^2$ gives,
$$\frac{1}{\tau^2} \langle F(z +\tau \delta z)-F(z),z +\tau \delta z-z \rangle > - \frac{1}{\tau^2}\|F(z +\tau \delta z)-F(z)\|^2$$

Since the limits of both the LHS and RHS exist and the limit preserves inequalities, applying the limit $\tau \rightarrow 0$ to both sides we obtain,

$$\lim_{\tau \rightarrow 0} \frac{1}{\tau^2}\langle F(z +\tau \delta z)-F(z),\tau \delta z \rangle > - \lim_{\tau \rightarrow 0}\frac{1}{\tau^2} \|F(z +\tau \delta z)-F(z)\|^2$$

this gives,
$$\lim_{\tau \rightarrow 0} \langle \frac{F(z +\tau \delta z)-F(z)}{\tau},\delta z \rangle > - \lim_{\tau \rightarrow 0}\left\|{\frac{F(z +\tau \delta z)-F(z)}{\tau}}\right\|^2$$
which gives upon rearranging,
\begin{align}\label{eqn:operator-eigen}
\langle \nabla F(z) \delta z,\delta z \rangle + \|\nabla F(z) \delta z\|^2 \geq 0
\end{align}
We now prove that  $t \mapsto \|F(z(t))\|$ is non-increasing for the dynamics ~\eqref{sys:DE}. 
The dynamics are,
\begin{equation*}
\dot{u}(t) = -F(z(t)), \quad v(t) = z_0 + u(t), \quad z(t) - v(t) + F(z(t)) = \textbf{0}.
\end{equation*}
In this case, we can write $z(t) = (I + F)^{-1}v(t)$. Since $v(\cdot)$ is continuously differentiable, we have $x(\cdot)$ is also continuously differentiable. Define the function $g(t) = \frac{1}{2}\|v(t) - z(t)\|^2$. 
Then we have, 

\begin{equation*}
\dot{g}(t) = \langle \dot{v}(t) - \dot{z}(t), v(t) - z(t) \rangle = -\langle \dot{v}(t) - \dot{z}(t), \dot{v}(t) \rangle = -\|\dot{v}(t) - \dot{z}(t)\|^2 - \langle \dot{v}(t) - \dot{z}(t), \dot{z}(t) \rangle. 
\end{equation*}

Now,
$$\dot{v}(t) - \dot{z}(t) = \nabla F(z(t)) \dot{z}(t)$$
 choosing $z = z(t), \delta z = \dot z (t)$ in Eq.~\eqref{eqn:operator-eigen} we obtain,
\begin{align*}
 \dot{g}(t) = -(\|\nabla F (z(t)) \dot z(t) \|^2+\langle \nabla F(z(t)) \dot z,\dot z \rangle) < 0
\end{align*}
Thus $\|F(z(t))\|$ is decreasing and combined with Eq.~\eqref{eqn:min} we have, 
$$ \|F(z(t))\|^2 \leq O\pa{\frac{1}{t}}$$ which is the statement of the corollary.
\end{proof}